\documentclass[a4paper]{amsart}
\usepackage{amsmath,amsthm,amssymb}
\usepackage[pdftex]{graphicx}
\usepackage{mathrsfs}
\usepackage[dvipdfmx, usenames]{color}
\usepackage[all]{xy}
\usepackage{comment}
\usepackage{bm}

\newtheorem{thm}{Theorem}[section]
\newtheorem{lem}[thm]{Lemma}
\newtheorem{cor}[thm]{Corollary}
\newtheorem{prop}[thm]{Proposition}

\theoremstyle{definition}

\newtheorem{example}[thm]{Example}

\newtheorem{defn}[thm]{Definition}
\newtheorem{assump}[thm]{Assumption}

\theoremstyle{remark}

\newtheorem{rem}[thm]{Remark}        % \renewcommand{\therem}{}

\renewcommand{\theequation}{
\thesection.\arabic{equation}}
\numberwithin{equation}{section}
\def\XXint#1#2#3{{\setbox0=\hbox{$#1{#2#3}{\int}$}
\vcenter{\hbox{$#2#3$}}\kern-.5\wd0}}

\newcommand{\R}{\mathbb{R}}

\newcommand{\m}{\mathfrak{m}}
\newcommand{\Lip}{\mathsf{Lip}}
\newcommand{\Cpl}{\mathsf{Cpl}}
\newcommand{\di}{\mathsf{Diam}\,}
\newcommand{\supp}{\mathsf{supp}\,}

\newcommand{\Ch}{\mathsf{Ch}}
\newcommand{\vol}{\mathsf{vol}}
\newcommand{\RCD}{\mathsf{RCD}}
\let\CD\relax
\newcommand{\CD}{\mathsf{CD}}
\newcommand{\ul}{\underline}
\newcommand{\ol}{\overline}
\newcommand{\la}{\left\langle}
\newcommand{\ra}{\right\rangle}
\newcommand{\del}{\partial}

\newcommand{\lv}{\left\vert}
\newcommand{\rv}{\right\vert}
\newcommand{\lV}{\left\Vert}
\newcommand{\rV}{\right\Vert}

\newcommand{\Ent}{\mathsf{Ent}}
\newcommand{\Tan}{\mathsf{Tan}}
\newcommand{\EVI}{\mathsf{EVI}}

\newcommand{\Var}{\mathsf{Var}}
\newcommand{\AVR}{\mathsf{AVR}}
\newcommand{\Hess}{\mathrm{Hess}}

\newcommand{\calD}{\mathcal{D}}

\newcommand{\calH}{\mathcal{H}}

\newcommand{\calL}{\mathcal{L}}

\newcommand{\calP}{\mathcal{P}}

\newcommand{\calR}{\mathcal{R}}

\newcommand{\bbN}{\mathbb{N}}

\newcommand{\bbS}{\mathbb{S}}

\newcommand{\sfb}{\mathsf{b}}

\newcommand{\fraks}{\mathfrak{s}}

\DeclareMathOperator{\argmin}{\mathrm{argmin}}

\usepackage[abbrev,nobysame]{amsrefs}
\makeatletter
\def\@makefnmark{%
\leavevmode
\raise.9ex\hbox{\check@mathfonts
\fontsize\sf@size\z@\normalfont%
\@thefnmark}%
}
\makeatother
\title{Shannon's inequality on non-collapsing $\RCD$ spaces}
\author{Yu Kitabeppu}
\address{Kumamoto University}
\email[Yu Kitabeppu]{ybeppu@kumamoto-u.ac.jp}

\begin{document}
\maketitle
 \begin{abstract}
  We prove the Shannon's inequality on non-collapsing $\RCD(0,N)$ spaces. In the proof, we use the characterization of the $\EVI_{0,N}$-gradient flow of the relative entropy and the infinitesimal behavior of the heat kernel. Also we have a cone rigidity result. As an application, we have the so-called uncertainty principle inequality on such spaces. 
 \end{abstract}
%%%%%%%%%%%%%%%%%%%%%%%%%%%%%%%%%%%%%%%%%%%%%%%%%%%%%%%%%%%%%%%%%%%%%%%%%%%%%%%%%%%%%%%%%%%%%%%%%%%%%%
%
%  Section  :  Introduction
%
%%%%%%%%%%%%%%%%%%%%%%%%%%%%%%%%%%%%%%%%%%%%%%%%%%%%%%%%%%%%%%%%%%%%%%%%%%%%%%%%%%%%%%%%%%%%%%%%%%%%%%
\section{Introduction}
Shannon proved the following inequality; 
%%%%%%%%%%%%
 \begin{thm}[\cite{Sh}]
  For any nonnegative function $f\in L^1(\R^n)$ with $\lv x\rv^2f(x)\in L^1(\R^n)$ and with $\lV f\rV_1=1$, it holds that 
  \begin{align}
   -\int_{\R^n}f(x)\log f(x)\,dx\leq \frac{n}{2}\log\left(\frac{2\pi e}{n}\int_{\R^n}\lv x\rv^2f(x)\,dx\right).\label{eq:Sh}
  \end{align}
  The inequality is sharp since the Gaussian functions $G_t(x)=(4\pi t)^{-n/2}e^{-\lv x\rv^2/4t}$ attain the equality in (\ref{eq:Sh}). 
 \end{thm}
%%%%%%%%%%%%
In the information theory, the inequality (\ref{eq:Sh}) has important meaning, see \cite{Sh}. Shannon's inequality is extended to various settings. For example, Shannon's inequality for R\'enyi entropy is known in $\R^n$ \cite{Suguro}. 
One of the main purpose in this article is to prove a similar inequality in more general settings. 
More precisely, we consider the inequality (\ref{eq:Sh}) on so-called \emph{$\RCD$ spaces}. 
\smallskip 
\par $\RCD$ space is a metric measure space equipped with the Riemannian structure and of Ricci curvature bounded from below and dimension bounded from above in a synthetic sense. Lott-Villani (\cite{LV} for $K\in \R$ and $N=\infty$ and for $K=0$ and finite $N$) and Sturm (\cite{St1,St2} for $K\in \R$ and $N\in(1,\infty]$) defined the $\CD(K,N)$ condition (curvature-dimension condition) on metric measure spaces independently. They define curvature and dimension bound in terms of the convexity of the relative entropy ($N=\infty$) and that of the R\'enyi entropy ($N<\infty$) or that of a class of functionals on the space of probability measures, see \cite{LV,St1,St2}.
 Bacher-Sturm defined the \emph{reduced curvature dimension condition} $\CD^*(K,N)$ in order to guarantee the localization property of the curvature condition \cite{BaSt}. They also considered the tensorial property under the non-branching assumption. But that property needs the infinitesimally Hilbertian condition (see \cite{Stcor}). $\RCD(K,\infty)$ condition is defined by Ambrosio-Gigli-Savar\'e \cite{AGSmm} (for compact metric measure spaces) and by Ambrosio-Gigli-Mondino-Rajala \cite{AGMR} (for metric measure spaces with $\sigma$-finite measures), those which spaces satisfy $\CD(K,\infty)$ condition and the infinitesimally Hilbertian condition which is defined by Gigli (\cite{G}). For finite $N$, Erbar-Kuwada-Sturm \cite{EKS} and Ambrosio-Mondino-Savar\'e \cite{AMSnon} defined $\RCD^*(K,N)$ condition independently. Roughly speaking $\RCD^*$ condition is combining infinitesimally Hilbertian condition and $\CD^*$ ($\CD^*$ condition is equivalent to $\CD^e$ condition (Definition \ref{def:CDe}) under infinitesimally Hilbertian \cite{EKS}). However it proves the equivalence between $\CD$ condition and $\CD^*$ condition under the essentially non-branching condition with finite measure by Cavalletti-Milman \cite{CMglobal}, with locally finite measure by Li \cite{Li}. Since $\RCD^*(K,N)$ condition implies the essentially non-branching property \cite{RS}, $\RCD(K,N)$ condition has complete meaning. In the present paper, we mainly focus on \emph{non-collapsing $\RCD$ spaces}. A metric measure space $(X,d,\m)$ is called a non-collapsing $\RCD(K,N)$ space if it satisfies $\RCD(K,N)$ condition and $\m=\calH^N$, here $\calH^N$ is the $N$-dimensional Hausdorff measure. The name ``non-collapsing" comes from the theory of converging Riemannian manifolds with bounded Ricci curvature by Cheeger-Colding \cite{CC1,CC2,CC3}. See section 2 for the precise definitions.   
\smallskip
\par We consider the inequality (\ref{eq:Sh}) as an inequality of probability measures. For a probability measure $\nu=\rho\calL^n\ll \calL^n$ with finite second moment, the \emph{relative entropy} of $\nu$ is defined by 
\begin{align}
 \Ent_{\calL^n}(\nu):=\int_{\{\rho>0\}}\rho\log \rho\,d\calL^n\notag
\end{align}
and the \emph{variation} of $\nu$ is defined as 
\begin{align}
 \Var(\nu):=\inf_{y\in \R^n}\int_{\R^n}\lv x-y\rv^2\,\nu(dx).\notag
\end{align}
Then the inequality (\ref{eq:Sh}) is rewritten by 
\begin{align}
 -\Ent_{\calL^n}(f\calL^n)\leq \frac{n}{2}\log\left(\frac{2\pi e}{n}\Var(f\calL^n)\right).\notag
\end{align}
%%%%%%%%%%%%
 \begin{thm}[Theorem \ref{thm:shannonRCD}]
  Let $(X,d,\calH^N)$ be a non-collapsing $\RCD(0,N)$ space and $\nu\in \calP_2(X)$ a probability measure on $X$ satisfying $\sfb(\nu)\cap \calR_N\neq\emptyset$. Then 
  \begin{align}
   -\Ent_{\calH^N}(\nu)\leq \frac{N}{2}\log\left(\frac{2\pi e}{N}\Var(\nu)\right)\label{eq:measuredSh}
  \end{align}
  holds. 
 \end{thm}
%%%%%%%%%%%%
In the statement, $\sfb(\nu)$ is the set of barycenters of $\nu$ and $\calR_N$ is the $N$-dimensional regular set, see section 2 and 3 for the definitions. As long as the author known, the Shannon inequality is new even for Riemannian manifolds with nonnegative Ricci curvature. 
\medskip
\par On $\R^n$, the equality of (\ref{eq:Sh}) is achieved by a function $f$ if and only if $f$ is the Gaussian function. In $\RCD$ spaces, the equality of (\ref{eq:Sh}) forces the shape of $\RCD$ spaces into cone if the equality is achieved by the gradient flow of the relative entropy functional, called the heat flow in the present paper, instead of the Gaussian function. We give the following rigidity result.  
%%%%%%%%%%%%
 \begin{thm}[Cone rigidity Theorem \ref{thm:conerigid}]
  Assume the equality in (\ref{eq:measuredSh}) holds for the heat flow, that is, 
  \begin{align}
   -\Ent_{\calH^N}(\mu^{x_0}_t)=\frac{N}{2}\log\left(\frac{2\pi e}{N}W_2^2(\delta_{x_0},\mu_t^{x_0})\right)\notag
  \end{align}
  holds, then $(X,d,\m)$ is a cone. 
 \end{thm}
%%%%%%%%%%%%
See section 4 for the precise definitions and assumption. So far the author does not know whether the $\RCD(0,N)$ space in which the equality (\ref{eq:measuredSh}) is supported by another probability measure different from the heat flow has to be a cone or not (see Remark \ref{rem:attainnotheat}). We also deal with an almost rigidity result in section 4. 
\smallskip
\par As an application of Shannon's inequality, we give the Heisenberg-Pauli-Weyl uncertainty principle inequality on $\RCD$ spaces. 
%%%%%%%%%%%%
 \begin{thm}[Theorem \ref{thm:UPRCD}]
  Let $(X,d,\calH^N)$ be a non-collapsing $\RCD(0,N)$ space. Assume $\displaystyle{\AVR_{\calH^N}:=\lim_{r\rightarrow \infty}\frac{\calH^N(B_r(x))}{\omega_Nr^N}>0}$, then for a probability measure $\nu\in\calP_2(X)$ with $\sfb(\nu)\cap\calR_N\neq\emptyset$, 
  \begin{align}
   I_{\calH^N}(\nu)^{1/2}\Var(\nu)^{1/2}\geq N\AVR_{\calH^N}^{1/N}\notag
  \end{align}
  holds. Here $I_{\calH^N}$ denotes the Fisher information (see after (\ref{eq:ElogSobmeas}))
 \end{thm}
%%%%%%%%%%%%
%%%%%%%%%%%%
 \begin{rem}
  Uncertainty principle inequality is extended to many situation, see for instance, \cite{DT,E,HKZ,Kr,KO,MM,OST}. Also see \cite{FS} as historical reference for the uncertainty principle.  
 \end{rem}
%%%%%%%%%%%%
%%%%%%%%%%%%%%%%%%%%%%%%%%%%%%%%%%%%%%%%%%%%%%%%%%%%%%%%%%%%%%%%%%%%%%%%%%%%%%%%%%%%%%%%%%%%%%%%%%
%
%   Section : Preliminaries 
% 
%%%%%%%%%%%%%%%%%%%%%%%%%%%%%%%%%%%%%%%%%%%%%%%%%%%%%%%%%%%%%%%%%%%%%%%%%%%%%%%%%%%%%%%%%%%%%%%%%%
\section{Preliminaries}
Let $(Y,d_Y)$ be a metric space. 
%For a continuous curve $\gamma:[0,1]\rightarrow Y$, the length of $\gamma$ is defined by 
%\begin{align}
% L(\gamma):=\sup\left\{\sum_{i=1}^nd_Y(\gamma_{t_{i-1}},\gamma_{t_i})\;;\;0=t_0<t_1<t_2<\cdots<t_{n-1}<t_n=1,\,n\in\bbN\right\}.\notag
%\end{align}
A continuous curve $\gamma:[0,1]\rightarrow Y$ is called a \emph{geodesic} connecting from $y_0$ to $y_1$ if $\gamma_0=y_0$, $\gamma_1=y_1$, and $d_Y(\gamma_s,\gamma_t)=\lv s-t\rv d_Y(\gamma_0,\gamma_1)$ holds for any $s,t\in [0,1]$. 
%%%%%%%%%%%%%%%%%%%%%%%%%%%%%%%%%%%%%%%%%%%%%%%%%%%%%%%%%%%%%%%%%%%%%%%%%%%%%%%%%%%%%%%%%%%%%%%%%%%%%%%%%%%%%%%%%%%%%%%%%%%%%%%%%%%%%%%%%%%%%%%%%%%
\subsection{$(K,N)$-convexity}
%%%%%%%%%%%
 \begin{defn}
  For $\kappa\in\R$, $\theta\geq0$, we define  
  \begin{align}
   \fraks_{\kappa}(\theta):=\begin{cases}
   \frac{1}{\sqrt{\kappa}}\sin(\sqrt{\kappa}\theta)&\text{if } \kappa>0,\\
   \theta&\text{if }\kappa=0,\\
   \frac{1}{\sqrt{-\kappa}}\sinh(\sqrt{-\kappa}\theta)&\text{if }\kappa<0.
   \end{cases}\notag
  \end{align}
 And for $t\in [0,1]$, we also define  
  \begin{align}
   \sigma^{(t)}_{\kappa}(\theta):=\begin{cases}
   \frac{\fraks_{\kappa}(t\theta)}{\fraks_{\kappa}(\theta)}&\text{if }\kappa\theta^2\neq 0\text{ and }\kappa\theta^2<\pi^2,\\
   t&\text{if }\kappa\theta^2=0,\\
   +\infty&\text{if }\kappa\theta^2\geq\pi^2. 
   \end{cases}\notag
  \end{align}
 \end{defn}
%%%%%%%%%%%
Let $(X,d)$ be a complete separable metric space and $S:X\rightarrow (-\infty,\infty]$ a functional on $X$. We denote  
\begin{align}
 \calD(S):=\left\{x\in X\;;\; S(x)<\infty\right\}.\notag
\end{align}
For $N\in(1,\infty)$, we define  
\begin{align}
 U_N(x):=\exp\left(-\frac{1}{N}S(x)\right).\notag
\end{align} 
%%%%%%%%%%%
 \begin{defn}[\cite{EKS}*{Definition 2.7}]
  For $K\in\R$, $N\in(1,\infty)$, $S$ is said to be \emph{$(K,N)$-convex} if the following holds; For any $x_0,x_1\in\calD(S)$, there exists a geodesic $\gamma:[0,1]\rightarrow X$ connecting them such that 
  \begin{align}
   U_N(\gamma_t)\geq \sigma^{(1-t)}_{K/N}(d(x_0,x_1))\cdot U_N(x_0)+\sigma^{(t)}_{K/N}(d(x_0,x_1))\cdot U_N(x_1)\notag
  \end{align}
  is satisfied for any $t\in[0,1]$.  
 \end{defn}
%%%%%%%%%%%
 We need the following concept of the gradient flow of a function, depending on $K$ and $N$. 
%%%%%%%%%%%
 \begin{defn}[\cite{EKS}*{Definition 2.14}]
  Fix $K\in\R$, $N\in(1,\infty)$. Let $x:(0,\infty)\rightarrow \calD(S)$ be a locally absolutely continuous curve. $(x_t)_t$ is called an \emph{$\EVI_{K,N}$ gradient flow} of $S$ with the initial data $x_0$ if $\lim_{t\rightarrow+0}x_t=x_0$ and for any $z\in\calD(S)$, 
  \begin{align}
   \frac{d}{dt}\fraks_{K/N}\left(\frac{1}{2}d(x_t,z)\right)^2+K\cdot \fraks_{K/N}\left(\frac{1}{2}d(x_t,z)\right)^2\leq \frac{N}{2}\left(1-\frac{U_N(z)}{U_N(x_t)}\right)\label{eq:evikn}
  \end{align}
  holds a.e. $t>0$. 
 \end{defn}
%%%%%%%%%%%
For $K=0$, (\ref{eq:evikn}) is written by 
\begin{align}
 \frac{1}{4}\frac{d}{dt}d(x_t,z)^2\leq \frac{N}{2}\left(1-\frac{U_N(z)}{U_N(x_t)}\right).\notag
\end{align}
For $x_0\in\ol{\calD(S)}$, there exists at most one $\EVI_{K,N}$-gradient flow with the initial data $x_0$ (see Corollary 2.21 in \cite{EKS}).
%%%%%%%%%%%
 \begin{rem}
  It is known that if $S$ is $(K,N)$-convex and differentiable on a complete smooth Riemannian manifold, then a smooth curve $x_t$ satisfies 
  \begin{align}
   \dot{x_t}=-\nabla S(x_t)\notag
  \end{align} 
  if and only if $(x_t)$ is a $\EVI_{K,N}$ gradient flow of $S$. See Lemma 2.4 in \cite{EKS}. 
 \end{rem}
%%%%%%%%%%%
%For $\kappa\in\R$, $t\geq 0$, we define  
%\begin{align}
% e_{\kappa}(t):=\int_0^te^{\kappa s}\,ds=\begin{cases}
%  \frac{e^{\kappa t}-1}{\kappa}&\text{if }\kappa\neq 0,\\
%  t&\text{if }\kappa=0. 
% \end{cases}\notag
%\end{align}
%For a function $f:I\rightarrow\R$ on an interval $I$, 
%\begin{align}
% \frac{d^+}{dt}f(t):=\limsup_{h\downarrow0}\frac{f(t+h)-f(t)}{h}.\notag
%\end{align}
%$D\subset \calD(S)$ is said to be \emph{dense in energy} if for any $z\in\calD(S)$, there exists a sequence $(z_n)\subset D$ such that $d(z_n,z)\rightarrow0$ and $S(z_n)\rightarrow S(z)$. 
%The following holds. 
%\begin{prop}[\cite{EKS}*{Proposition. 2.18}]
% $D\subset \calD(S)$ is dense in energy and $(x_t)$ is a $\EVI_{K,N}$ gradient flow of $S$. Then for any $z\in D$, $t>0$  
% \begin{align}
%  \frac{d^+}{dt}\fraks_{K/N}\left(\frac{1}{2}d(x_t,z)\right)^2+K\cdot \fraks_{K/N}\left(\frac{1}{2}d(x_t,z)\right)^2\leq \frac{N}{2}\left(1-\frac{U_N(z)}{U_N(x_t)}\right).\notag
% \end{align} 
%\end{prop}
%%%%%%%%%%%%%%%%%%%%%%%%%%%%%%%%%%%%%%%%%%%%%%%%%%%%%%%%%%%%%%%%%%%%%%%%%%%%%%%%%%%%%%%%%%%%%%%%%%%%%%%%%%%%%%%%%%%%%%%%%%%%%%%%%%%%%%%%%%%%%%%%%%%
\subsection{$\CD^e$ spaces}
Let $(X,d)$ be a metric space. We denote the set of all Borel probability measures on $X$ by $\calP(X)$. Set 
\begin{align}
 \calP_2(X):=\left\{\mu\in\calP(X)\;;\;\int_{X}d^2(o,x)\,\mu(dx)<\infty\text{ for }{\exists}/{\forall}o\in X\right\}.\notag
\end{align}
For $\mu,\nu\in\calP_2(X)$, $\xi\in \calP(X\times X)$ is called a \emph{coupling} between $\mu$ and $\nu$ if it satisfies 
\begin{align}
 \begin{cases}
  \xi(A\times X)=\mu(A)\\
  \xi(X\times A)=\nu(A)
 \end{cases}\notag
\end{align}
for any Borel subset $A\subset X$. The set of all couplings between $\mu$ and $\nu$ is denoted by $\Cpl(\mu,\nu)$. The $L^2$-Wasserstein distance between $\mu,\nu\in\calP_2(X)$ is defined by 
\begin{align}
 W_2(\mu,\nu):=\inf\left\{\lV d\rV_{L^2(\xi)}\;;\;\xi\in\Cpl(\mu,\nu)\right\}.\notag
\end{align}
$(\calP_2(X),W_2)$ is called the $L^2$-Wasserstein space. It is known that $(\calP_2(X),W_2)$ is a complete separable geodesic metric space if so is $(X,d)$. The explicit value of the Wasserstein distance between two given measures is difficult to calculate. However $W_2(\mu,\delta_z)$, where $\delta_z$ is the Dirac measure at $z\in X$ is easy to calculate, in fact, it is given by 
\begin{align}
 W_2^2(\mu,\delta_z)=\int_Xd^2(x,z)\mu(dx).\notag
\end{align} 
\smallskip
\par We call a triplet $(X,d,\m)$ \emph{metric measure space} if $(X,d)$ is a complete separable metric space and $\m$ is a locally finite Borel measure on $X$. Moreover, in this paper, we always assume that $\supp\m=X$ and there exists a constant $C>0$, $x_0\in X$ such that 
\begin{align}
 \int_Xe^{-Cd^2(x_0,x)}\,\m(dx)<\infty.\label{eq:integrable}
\end{align}
The \emph{relative entropy} functional $\Ent_{\m}:\calP_2(X)\rightarrow \R\cup\{+\infty\}$ is defined by 
\begin{align}
 \Ent_{\m}(\mu):=\begin{cases}
  \int_{\{\rho>0\}}\rho\log\rho\,d\m&\text{if }\mu=\rho\m\ll\m\\
  +\infty&\text{otherwise}.
 \end{cases}\notag
\end{align} 
Note that (\ref{eq:integrable}) guarantees $\Ent_{\m}(\mu)>-\infty$ for any $\mu\in\calP_2(X)$ (see for instance \cite{AGS}). 
%%%%%%%%%%%
 \begin{defn}[$\CD^e(K,N)$ space, \cite{EKS}]\label{def:CDe}
  Let $K\in\R$, $N\in(1,\infty)$. A metric measure space $(X,d,\m)$ is called a $\CD^e(K,N)$ space if $\Ent_{\m}$ is $(K,N)$-convex on $(\calP_2(X),W_2)$. 
 \end{defn}
%%%%%%%%%%% 
It is known that $\CD^e(K,N)$ spaces are geodesic spaces. 
%%%%%%%%%%%%%%%%%%%%%%%%%%%%%%%%%%%%%%%%%%%%%%%%%%%%%%%%%%%%%%%%%%%%%%%%%%%%%%%%%%%%%%%%%%%%%%%%%%%%%%%%%%%%%%%%%%%%%%%%%%%%%%%%%%%%%%%%%%%%%%%%%%%
\subsection{Infinitesimal Hilbertianity}
Let $(X,d,\m)$ be a metric measure space. For any locally Lipschitz function $f:X\rightarrow \R$, define 
\begin{align}
 \Lip f(x):=\limsup_{y\rightarrow x}\frac{\lv f(x)-f(y)\rv}{d(x,y)}\notag
\end{align}
if $x$ is not isolated, $\Lip f(x)=0$ otherwise. For any $f\in L^2(\m)$, we define the \emph{Cheeger energy} of $f$ as 
\begin{align}
 \Ch(f):=\frac{1}{2}\inf\left\{\liminf_{n\rightarrow \infty}\int_X(\Lip f_n)^2\,d\m\;;\;f_n:\text{loc. Lipschitz, }f_n\xrightarrow{L^2} f\right\}.\label{eq:CE}
\end{align}
For $f\in \calD(\Ch)=\{g\in L^2(\m)\;;\;\Ch(g)<\infty\}$, there exists an $L^2$-function $\lv \nabla f\rv$ such that 
\begin{align}
 \Ch(f)=\frac{1}{2}\int_X\lv \nabla f\rv^2\,d\m\notag
\end{align}
holds. The $L^2$ function $\lv \nabla f\rv$ is also characterized by the least Borel function $G:X\rightarrow [0,\infty]$ satisfying 
\begin{align}
 \int\lv f(\gamma_1)-f(\gamma_0)\rv\,d\pi(\gamma)\leq \int\int_0^1G(\gamma_s)\lv\dot{\gamma}_s\rv\,dsd\pi(\gamma)\notag
\end{align}
for any test plan $\pi\in \calP(C([0,1];X))$ and determined up to $\m$-a.e. see \cite {AGS}. It coincides with various concepts of gradient under mild assumptions (relaxed gradient, Newtonian and Cheeger's ones. See \cite{AGS,G,GPbook}). Note that we do not assert the existence of the gradient $\nabla f$ of $f$, but this is true by \cite{G}. 
%%%%%%%%%%%%
 \begin{defn}[\cite{G}]
  A metric measure space $(X,d,\m)$ is said to be \emph{infinitesimally Hilbertian} if for any $f,g\in\calD(\Ch)$, it holds that 
  \begin{align}
   \Ch(f+g)+\Ch(f-g)=2\left(\Ch(f)+\Ch(g)\right).\notag
  \end{align}
 \end{defn}
%%%%%%%%%%%%
It is known that $(X,d,\m)$ is infinitesimally Hilbertian if and only if the Sobolev space $W^{1,2}(X,d,\m):=\calD(\Ch)$ equipped with the norm $\lV f\rV_{1,2}^2:=\lV f\rV_2^2+2\Ch(f)$ is a Hilbert space. In fact, if $(X,d,\m)$ is infinitesimally Hilbertian, then for any $f,g\in W^{1,2}$, there exists a canonical $L^1$ function $\la \nabla f,\nabla g\ra$ such that 
\begin{align}
 \frac{1}{2}(\Ch(f+g)-\Ch(f-g))=\int_X\la \nabla f,\nabla g\ra\,d\m\notag
\end{align}
holds. 
\medskip
\par Later, we need the concept of the measure-valued Laplacian. Hence here we give some notions. Assume a metric measure space $(X,d,\m)$ is infinitesimally Hilbertian and proper (any closed bounded set is compact). Let $\Omega$ be an open subset of $X$. We define the local Sobolev space by 
\begin{align}
 W_{loc}^{1,2}(\Omega):=\left\{u\in L^2_{loc}(\Omega)\;;\;u\eta\in W^{1,2}(X),\;\text{for every }\eta\in\mathsf{LIP}_c(\Omega) \right\},\notag
\end{align} 
where $\mathsf{LIP}_c(\Omega)$ is the set of all Lipschitz functions with compact supports included in $\Omega$.   
%%%%%%%%%%%%%
 \begin{defn}[\cite{CMnew}, cf. \cite{GV,G}]
  Let $\Omega\subset X$ be an open subset. We say that $u\in W^{1,2}_{loc}(\Omega)$ belongs to the domain of the measured-valued Laplacian $\calD({\bf \Delta},\Omega)$ if there exists a Radon measure ${\bf \Delta}\vert_{\Omega}u$ in $\Omega$ such that 
  \begin{align}
   -\int_{\Omega}\la \nabla f,\nabla u\ra\,d\m=\int_{\Omega}f\,d{\bf \Delta}\vert_{\Omega}u\notag
  \end{align}
  holds for every $f\in \mathsf{LIP}_c(\Omega)$. When $\Omega=X$, we write the measure-valued Laplacian ${\bf \Delta}$ and the domain of that by $\calD({\bf \Delta})$. 
 \end{defn}
%%%%%%%%%%%%%
If ${\bf \Delta}\vert_{\Omega}u=h\m\ll\m$ for $h\in L^2_{loc}(\Omega)$, then we denote $\Delta\vert_{\Omega}u=h$ and $u\in\calD(\Delta)$. Moreover when $u,h\in L^2$, $\Delta$ is a linear operator by the infinitesimal Hilbertianity, that also coincides with the generator of the Dirichlet form associated with $2\Ch$.  
%%%%%%%%%%%%%%%%%%%%%%%%%%%%%%%%%%%%%%%%%%%%%%%%%%%%%%%%%%%%%%%%%%%%%%%%%%%%%%%%%%%%%%%%%%%%%%%%%%%%%%%%%%%%%%%%%%%%%%%%%%%%%%%%%%%%%%%%%%%%%%%%%%%
\subsection{$\RCD$ space}
In this subsection we always assume $K\in\R$, $N\in(1,\infty)$. The following definition is equivalent to that of $\RCD^*(K,N)$ in \cite{EKS}. 
 %%%%%%%%%%%% 
  \begin{defn}
   We call a metric measure space $(X,d,\m)$ \emph{$\RCD(K,N)$ space} if it is infinitesimally Hilbertian and satisfies the $\CD^e(K,N)$ condition. 
  \end{defn}
 %%%%%%%%%%%%
 %%%%%%%%%%%%
  \begin{rem}
   Under the $\RCD(K,N)$ condition, (\ref{eq:integrable}) is satisfied (\cite{EKS}). But for simplicity, we add (\ref{eq:integrable}) to the definition of metric measure space. 
  \end{rem}
 %%%%%%%%%%%%
 We give the following important examples. 
 %%%%%%%%%%%%
  \begin{example}\label{ex:RCD}
  \begin{enumerate}
   \item Let $(M^n,g)$ be a complete Riemannian manifold, $f:M\rightarrow \R$ a $C^2(M)$ function, $d_g$ the Riemannian distance function, and $\vol_g$ the Riemannian volume measure on $M$. Set $\m:=e^{-f}\vol_g$. Then the metric measure space $(M,d_g,\m)$ satisfies $\RCD(K,N)$ condition for $N>n$ if and only if 
   \begin{align}
    \mathrm{Ric}_N:=\mathrm{Ric}_g+\Hess_f-\frac{df\otimes df}{N-n}\geq Kg\notag
   \end{align}
   holds. For $N=n$, the $\RCD(K,n)$ condition is equivalent to $df=0$ and $\mathrm{Ric}_g\geq K$. 
   \item Let $\{(X_i,d_i,\m_i)\}_i$ be a family of $\RCD(K_i,N)$ spaces. For $x_i\in X_i$, assume $\m_i(B_1(x_i))=1$, $K_i\rightarrow K$ and $(X_i,d_i,\m_i,x_i)\xrightarrow{pmG}(X_{\infty},d_{\infty},\m_{\infty},x_{\infty})$ as $i\rightarrow \infty$, where $\xrightarrow{pmG}$ means the pointed measured Gromov convergence (see \cite{GMS}). Then $(X_{\infty},d_{\infty},\m_{\infty})$ satisfies the $\RCD(K,N)$ condition. Moreover a family of $\RCD(K,N)$ spaces with the normalized measures is precompact with respect to the pmG-convergence. 
   \end{enumerate}
  \end{example}
 %%%%%%%%%%%%
 We list geometric and analytic properties of $\RCD$ spaces in the following; Let $(X,d,\m)$ be an $\RCD(K,N)$ space. 
 \begin{enumerate}
  \item[(a)] (Bishop-Gromov inequality \cite{St2,EKS}): 
  \begin{align}
   \frac{\m(B_R(x))}{\m(B_r(x))}\leq \frac{\int_0^R\fraks_{K/N}(t)^N\,dt}{\int_0^r\fraks_{K/N}(t)^N\,dt}\notag
  \end{align}
  holds for any $0<r\leq R$, $x\in X$. Especially $(X,d)$ is proper. 
  \item[(b)] (Laplacian comparison \cite{G}): Here we take $K=0$. For fixed $x_0\in X$, $\frac{d_{x_0}^2}{2}\in\calD({\bf\Delta})$ and 
  \begin{align}
   {\bf \Delta}\frac{d_{x_0}^2}{2}\leq N\m\label{eq:Lcomp}
  \end{align}
  holds, where $d_{x_0}(\cdot):=d(x_0,\cdot)$. Moreover ${\bf \Delta}d_{x_0}^2$ is decomposed into the sum of two Radon measures, one is absolutely continuous to $\m$ and another is singular to $\m$, see more precise representation in \cite{CMnew}. For our purpose, we only need to know that ${\bf \Delta}d_{x_0}^2=\text{regular part}+\text{singular part}$. 
  %\item[(c)] (Splitting theorem): Assume $(X,d,\m)$ is $\RCD(0,N)$ space and contains an isometric embedding $\gamma:\R\rightarrow X$, that is called a straight line, then there exists an $\RCD(0,N-1)$ space $(Y,d_Y,\m_Y)$ such that 
  %\begin{align}
  % (X,d,\m)\simeq \left(Y\times \R,\sqrt{d_Y^2+d_E^2},\m_Y\otimes \calL^1\right).\notag
  %\end{align}
  %If $N\in[1,2)$, $Y$ is just a point.  
  \item[(c)] (Essential dimension \cite{BS}): For an $\RCD(K,N)$ space $(X,d,\m)$ and $x\in X$, 
  \begin{align}
   (X,d_r,\m_{x,r}):=(X,r^{-1}d,\m(B_r(x))^{-1}\m)\notag
  \end{align}
  is an $\RCD(r^2K,N)$ space. Hence there exists a sequence $(r_i)_i$ with $r_i\downarrow 0$ such that $(X,d_{r_i},\m_{x,r_i})$ converges to an $\RCD(0,N)$ space by (2) in Example \ref{ex:RCD}. Define the set of \emph{tangent cones} at $x$ by 
  \begin{align}
   &\Tan(X,x)\notag\\
   &:=\left\{(Y,d_Y,\m_Y,y)\;;\;~^{\exists}r_i\downarrow 0,\text{ s.t. }(X,d_{r_i},\m_{x,r_i},x)\xrightarrow{pmG} (Y,d_Y,\m_Y,y)\right\}.\notag
  \end{align} 
  Note that $\Tan(X,x)$ is the set of isomorphism classes of pointed metric measure spaces, where two pointed metric measure spaces $(X,d_X,\m_X,x)$ and $(Y,d_Y,\m_Y,y)$ is isomorphic to each other if there exists an isometry $f:(X,d_X)\rightarrow (Y,d_Y)$ with $f(x)=y$ such that $f_{\#}\m_X=\m_Y$ ($f_{\#}\m_X$ denotes the push-forward measure on $Y$). For an integer $k\in\bbN$, the $k$-dimensional regular set is defined as 
  \begin{align}
   \calR_k:=\left\{x\in X\;;\; \Tan(X,x)=\{(\R^k,d_E,\ul{\calL}^k,o^k)\}\right\},\notag
  \end{align}
  here $o^k$ is the origin in $\R^k$ and $\ul{\calL}^k:=\calL^k(B_1(o^k))^{-1}\calL^k$ is the normalized $k$-dimensional Lebesgue measure on $\R^k$. It is known that there exists an integer $1\leq k\leq N$ such that 
  \begin{align}
   \m\left(X\setminus \calR_k\right)=0.\notag
  \end{align}
  The above integer $k$ is denoted by $\dim_{ess}(X,d,\m)$ and called the \emph{essential dimension}. It is known that $\m\vert_{\calR_k}\ll \calH^k$ (see \cite{AHT,DMR,GP,KM}).
  \item[(d)] (Heat kernel \cite{AGSmm}): The $L^2$-Laplacian $\Delta$ is a self-adjoint linear operator on $L^2(\m)$. The associated semigroup $T_t$ is described by the \emph{heat kernel} $p$, that is, 
  \begin{align}
   T_tf(x)=\int_Xf(y)p(x,y,t)\m(dy)\notag
  \end{align}
  holds for $f\in L^2(\m)$. The heat semigroup can be seen as the gradient flow of the Cheeger energy, which is called the \emph{heat flow}. The heat kernel satisfies the heat equation; for $t>0$, it holds 
  \begin{align}
   \del_t p(x,y,t)=\Delta_y p(x,y,t)\notag
  \end{align}
  in weak sense. The heat kernel $\tilde{p}$ on the rescaled metric measure space $(X,d_r,C\m)$ which is an $\RCD(r^2K,N)$ space is 
  \begin{align}
   \tilde{p}(x,y,t)=C^{-1}p(x,y,r^{2}t).\label{eq:heatrescaled}
  \end{align}
  For $\RCD(0,N)$ spaces, the \emph{heat kernel estimate} is known; For any $\varepsilon>0$, there exists a constant $C=C(\varepsilon)>0$ such that 
  \begin{align}
   \frac{C(\varepsilon)^{-1}}{\m\left(B_{\sqrt{t}}(y)\right)}\exp\left(-\frac{d^2(x,y)}{(4-\varepsilon)t}\right)\leq p(x,y,t)\leq \frac{C(\varepsilon)}{\m\left(B_{\sqrt{t}}(y)\right)}\exp\left(-\frac{d^2(x,y)}{(4+\varepsilon)t}\right)\label{eq:hkestimate}
  \end{align}
  for all $t>0$, $x,y\in X$ (\cite{JLZ}). 
  \item[(e)] (Coincidence with two flows \cite{AGS,EKS}): Let $\nu=f\m\in\calP_2(X)$ be a Borel probability measure which is absolutely continuous to $\m$ with the density function $f\in L^2(\m)$. Then $\nu_t:=(T_tf)\m$ is the $\EVI_{K,N}$-gradient flow of $\Ent_{\m}(\cdot)$ starting from $\nu$. Moreover, though $\delta_x$ ($x\in X$) is not absolutely continuous to $\m$, $t\mapsto p(x,y,t)\m(dy)$ is also the $\EVI_{K,N}$-gradient flow of $\Ent_{\m}$ starting from $\delta_x$. Note that $t\mapsto U_N(\nu_t)$ is non-decreasing. 
  \item[(f)] (Non-collapsing space \cite{DG,Hnew,BGHZ}): An $\RCD(K,N)$ space $(X,d,\m)$ is called a \emph{non-collapsing $\RCD(K,N)$ space} if $N\in\bbN$ and $\dim_{ess}(X,d,\m)=N$. One of the most remarkable properties on non-collapsing $\RCD$ space is the measure $\m$ coincides with the $N$-dimensional Hausdorff measure up to multiplying a positive constant. On every non-collapsing $\RCD$ space, $x\in\calR_N$ if and only if  
  \begin{align}
   \lim_{t\downarrow0}\frac{\calH^N(B_t(x))}{\omega_Nt^N}=1,\notag
  \end{align}
where $\omega_N:=\pi^{N/2}/\Gamma(1+N/2)$ is the volume of $N$-dimensional Euclidean unit balls \cite{BGHZ}.  
  \item[(g)] (Regularity of $U_N(\mu_t)$ \cite{EKS}): Let $\mu_t$ be the $\EVI_{K,N}$-gradient flow of $\Ent_{\m}$ with the initial data $\mu_0$. Then the map $t\mapsto U_N(\mu_t)$ is concave in $(0,+\infty)$. Accordingly it is locally absolutely continuous. 
 \end{enumerate}
 We give the following proposition that we need later. 
 %%%%%%%%%%
 \begin{prop}
  Let $(X,d,\calH^N)$ be a non-collapsing $\RCD(0,N)$ space and $x\in\calR_N$. We denote the heat kernel on $X$ by $p$. Define the constant 
  \begin{align}
   \alpha_t:=\frac{\omega_N}{\calH^N(B_{t^{1/2}}(x))}.\label{eq:alpha}
  \end{align}
  Then we have 
  \begin{align}
   \int_X\frac{1}{\alpha_t}p(x,y,t)\log\frac{1}{\alpha_t}p(x,y,t)\,(\alpha_t\calH^N)(dy)\rightarrow -\frac{N}{2}\log(4\pi e)\label{eq:limitEnt}
  \end{align}
  as $t\downarrow 0$. 
 \end{prop}
%%%%%%%%%%
 \begin{proof}
    Since $x\in\calR_N$, we have $\delta_x\rightarrow \delta_{0^N}$ ($0^N$ is the origin in $\R^N$) in $W_2$ (here, we use the $\lq\lq$extrinsic notion" of $pmG$ convergence. see \cite{GMS} for the precise meaning). $\alpha_t^{-1}p(x,y,t)$ coincides with $\tilde{p}(x,y,1)$ which is the heat kernel of the metric measure space $(X,t^{-1/2}d,\alpha_t\calH^N)$ by (\ref{eq:heatrescaled}). By Theorem 7.7 in \cite{GMS}, the limit of the left-hand side in (\ref{eq:limitEnt}) coincides with $\Ent_{\calL^N}(p^{\R^N}(0^N,\cdot,1)\calL^N)$. Hence by the simple calculation, we obtain the limit is $\displaystyle{-\frac{N}{2}\log(4\pi e)}$. 
   \end{proof}
%%%%%%%%%%
%%%%%%%%%%%%%%%%%%%%%%%%%%%%%%%%%%%%%%%%%%%%%%%%%%%%%%%%%%%%%%%%%%%%%%%%%%%%%%%%%%%%%%%%%%%%%%%%%%
%
%  Section : Shannon's inequality
% 
%%%%%%%%%%%%%%%%%%%%%%%%%%%%%%%%%%%%%%%%%%%%%%%%%%%%%%%%%%%%%%%%%%%%%%%%%%%%%%%%%%%%%%%%%%%%%%%%%%
\section{Shannon's inequality}
%%%%%%%%%%%%%%%%%%%%%%%%%%%%%%%%%%%%%%%%%%%%%%%%%%%%%%%%%%%%%%%%%%%%%%%%%%%%%%%%%%%%%%%%%%%%%%%%%%%%%%%%%%%%%%%%%%%%%%%%%%%%%%%%%%%%%%%%%%%%%%%%%%%
\subsection{Proof of Shannon's inequality on $\R^n$}
In this subsection, we give a proof of Shannon's inequality on $\R^n$ seen as an $\RCD(0,n)$ space. 
\begin{proof}
 We denote the heat kernel in $\R^n$ by $p^{\R^n}$. Let $f:\R^n\rightarrow \R$ be a nonnegative function with $\lV f\rV_1=1$ and with $\lv x\rv^2f(x)\in L^1(\R^n)$. Define $\nu:=f\calL^n$, accordingly $\nu\in\calP_2(\R^n)$. We take a point $\omega\in\R^n$ such that 
 \begin{align}
  \omega\in\argmin_{y\in\R^n}\int_{\R^n}\lv x-y\rv^2f(x)\,dx.\label{eq:baryRn}
 \end{align}
 The minimizers in the right-hand side in (\ref{eq:baryRn}) always exist and it is unique. %Though for generic metric space, minimizers may not be unique, but it does not matter in the proof later. 
 Define 
 \begin{align}
  \mu_t(dx):=p^{\R^n}(\omega,x,t)\calL^n(dx),\notag
 \end{align}
 which is the $\EVI_{0,n}$ gradient flow of $\Ent_{\calL^n}$ with the initial data $\delta_{\omega}$. Set 
 \begin{align}
  \theta^2:=\int_{\R^n}\lv x-\omega\rv^2\,\nu(dx)=W_2^2(\nu,\delta_{\omega})=W_2^2(\nu,\mu_0).\notag
 \end{align}
 Integrate (\ref{eq:evikn}) from 0 to $t$, then 
 \begin{align}
  &\frac{1}{4}W_2^2(\mu_t,\nu)-\frac{1}{4}W_2^2(\mu_0,\nu)\leq \int_0^t\frac{n}{2}\left(1-\frac{U_n(\nu)}{U_n(\mu_s)}\,ds\right)\notag\\
  &=\frac{nt}{2}-\frac{nU_n(\nu)}{2}\int_0^t\frac{1}{U_n(\mu_s)}ds=:\frac{nt}{2}-\frac{nU_n(\nu)}{2}F(t).\notag 
 \end{align}
 Then we have 
 \begin{align}
  U_n(\nu)&\leq \frac{1}{F(t)}\left(t+\frac{1}{2n}\left(\theta^2-W_2^2(\mu_t,\nu)\right)\right)\label{eq:thetaW2}\\
  &\leq \frac{1}{F(t)}\left(t+\frac{1}{2n}\theta^2\right)\label{eq:UnRn}
 \end{align}
 holds for a.e. $t>0$. Since both side in (\ref{eq:UnRn}) are continuous in $t$, (\ref{eq:UnRn}) actually holds for any $t>0$. By a simple calculation, we have 
 \begin{align}
  U_n(\mu_t)=\sqrt{4\pi e t},\quad F(t)=\sqrt{\frac{t}{\pi e}}.\notag
 \end{align}
 Substitute $t=\theta^2/2n$ to (\ref{eq:UnRn}) leads 
 \begin{align}
  U_n(\nu)\leq \sqrt{\frac{2n\pi e}{\theta^2}}\left(\frac{\theta^2}{2n}+\frac{\theta^2}{2n}\right)=\sqrt{\frac{2\pi e}{n}\theta^2}.\notag
 \end{align}
 It is what we need. 
\end{proof}
%%%%%%%%%%%%%%%%%%%%%%%%%%%%%%%%%%%%%%%%%%%%%%%%%%%%%%%%%%%%%%%%%%%%%%%%%%%%%%%%%%%%%%%%%%%%%%%%%%%%%%%%%%%%%%%%%%%%%%%%%%%%%%%%%%%%%%%%%%%%%%%%%%%
\subsection{Proof of Shannon's inequality on non-collapsing $\RCD(0,N)$ spaces}
In this subsection, we always assume $(X,d,\m)$ is an $\RCD(0,N)$ space which is not one point. 
%%%%%%%%%%
 \begin{lem}\label{lem:theta}
  Take a point $o\in X$. Define $\mu_t(dy):=p(o,y,t)\m(dy)$ with $\mu_0=\delta_o$, and $\theta_t^2:=W_2^2(\delta_o,\mu_t)$. Then 
  \begin{align}
   \theta^2_t\leq 2Nt.\notag
  \end{align}
 \end{lem}
%%%%%%%%%%
 \begin{proof}
  Note that $\theta_0^2=0$ and 
  \begin{align}
   \theta_t^2=\int_Xd^2(o,y)p(o,y,t)\m(dy).\notag
  \end{align}
  Then 
  \begin{align}
   \theta_t^2=\theta_t^2-\theta_0^2=\int_0^t\frac{d}{ds}\theta_s^2\,ds=2\int_0^t\frac{1}{2}\frac{d}{ds}W_2^2(\delta_o,\mu_s)\,ds.\notag
  \end{align}
  We have 
  \begin{align}
   \frac{1}{2}\frac{d}{dt}W_2^2(\delta_o,\mu_t)&=\int_X\frac{d_o^2(y)}{2}\frac{\del}{\del t}p(o,y,t)\m(dy)\notag\\
   &=\int_X\frac{d_o^2(y)}{2}\Delta_yp(o,y,t)\m(dy)\notag\\
   &=\int_Xp(o,y,t)d{\bf \Delta}\frac{d_o^2}{2}(y)\label{eq:Lapheat}\\
   &\leq N\int_Xp(o,y,t)\m(dy)\notag\\
   &=N\notag
  \end{align}
  by (\ref{eq:Lcomp}). Hence we obtain 
  \begin{align}
   \theta_t^2\leq 2N\int_0^t\,ds=2Nt.\notag
  \end{align}
 \end{proof} 
%%%%%%%%%%
%%%%%%%%%%
 \begin{rem}
  The heat kernel is a Lipschitz function but not compact supported in non-compact setting. However by the cut-off argument, we are able to justify (\ref{eq:Lapheat}) even for such situation. See \cite{MN,AMSbe,GMos} for the construction of smooth cut-off functions.  
 \end{rem}
%%%%%%%%%%
%%%%%%%%%%
 \begin{lem}\label{lem:key1}
  For any $x\in X$, we define $\mu_t^x(dy):=p(x,y,t)\m(dy)$ and $F^x(t):=\int_0^tU_N(\mu_s^x)^{-1}\,ds$. Then the map $t\mapsto F^x(t)/\sqrt{t}$ is monotone increasing and the limit $\lim_{t\rightarrow +0}F^x(t)/\sqrt{t}=:C_0=C_0(x)\geq 0$ exists. The constant $C_0\geq 0$ is the largest constant $C$ satisfying  
  \begin{align}
   F^x(t)\geq C\sqrt{t}\label{eq:lC}
  \end{align}
  for any $t>0$.  
 \end{lem}
%%%%%%%%%%
 \begin{proof}
  Along the proof of the case in $\R^n$, we obtain 
  \begin{align}
   U_N(\mu_t^x)\leq \frac{1}{F^x(t)}\left(t+\frac{\theta_t^2}{2N}\right),\notag
  \end{align}
  where $\theta_t^2:=W_2^2(\mu_t^x,\delta_x)$. By the fact (g), we know $F^x$ is locally absolutely continuous in $(0,\infty)$. Then $F^x(t)$ is differentiable a.e. $t>0$. Since $(F^x)'(t)=U_N(\mu_t^x)^{-1}$, $\theta_t^2\leq 2Nt$, we have $F^x(t)\leq 2t (F^x)'(t)$. Therefore it holds 
\begin{align}
 \left(\frac{F^x(t)}{\sqrt{t}}\right)'=\frac{(F^x)'(t)\sqrt{t}-F^x(t)\frac{1}{2\sqrt{t}}}{t}=\frac{(F^x)'(t)2t-F^x(t)}{2t^{3/2}}\geq 0. \notag
\end{align}   
Since $F^x$ is locally absolutely continuous, we have the monotonicity of the function $F^x(t)/\sqrt{t}$ a.e. $t>0$. Again by the locally absolutely continuity of both side, we have the conclusion. 
 \end{proof}
%%%%%%%%%%
%%%%%%%%%% 
 \begin{lem}\label{lem:key2}
  Let $(X,d,\calH^N)$ be a non-collapsing $\RCD(0,N)$ space and $x\in\calR_N$. Then the constant $C_0(x)$ appearing in Lemma \ref{lem:key1} is equal to $(\pi e)^{-1/2}$. 
 \end{lem}
%%%%%%%%%%
 \begin{proof}
  The constant $\alpha_t$ is as in (\ref{eq:alpha}). For abbreviation, we denote $(X,d_{t^{1/2}},\alpha_t\calH^N)$ by $X_t$. we define $\mu_t^X(dy):=p(x,y,t)\calH^N(dy)$ and $\mu^{X_t}_1(dy)=p^{X_t}(x,y,1)(\alpha_t\calH^N)(dy)$. Then 
  \begin{align}
   &\Ent_{\calH^N}(\mu_t^X)=\int_Xp(x,y,t)\log p(x,y,t)\,\calH^N(dy)\notag\\
   &=\int_X\left(\frac{1}{\alpha_t}p(x,y,t)\right)\log\left(\frac{1}{\alpha_t}p(x,y,t)\right)\left(\alpha_t\calH^N\right)(dy)\notag\\
   &-\int_X\left(p(x,y,t)\right)\left(\log\frac{1}{\alpha_t}\right)\,\calH^N(dy)\notag\\
   &=\int_Xp^{X_t}(x,y,1)\log p^{X_t}(x,y,1)\,(\alpha_t\calH^N)(dy)-\log\frac{1}{\alpha_t}.\notag
  \end{align}
  Hence 
  \begin{align}
   \Ent_{\calH^N}(\mu_t^X)=\Ent_{\alpha_t\calH^N}(\mu_1^{X_t})-\log\left(\frac{\calH^N(B_{t^{1/2}}(x))}{\omega_Nt^{N/2}}\right)-\log t^{N/2}.\notag
  \end{align}
  Thus we obtain 
  \begin{align}
   &\exp\left(-\frac{1}{N}\Ent_{\calH^N}(\mu_t^X)\right)\notag\\
   &=\exp\left(-\frac{1}{N}\Ent_{\alpha_t\calH^N}(\mu_1^{X_t})\right)\times \left(\frac{\calH^N(B_{t^{1/2}}(x))}{\omega_Nt^{N/2}}\right)^{1/N}\times t^{1/2}.\notag
  \end{align}
  Since $x\in\calR_N$, combining (\ref{eq:limitEnt}) and the above calculation, we have 
  \begin{align}
   \lim_{t\rightarrow +0}\frac{U_N(\mu_t^X)}{t^{1/2}}=(4\pi e)^{1/2}.\notag
  \end{align}
  Finally, since $(F^x)'(t)=U_N(\mu_t^X)^{-1}$, we obtain 
  \begin{align}
   \lim_{t\rightarrow +0}\frac{F^x(t)}{t^{1/2}}=\frac{2}{(4\pi e)^{1/2}}=\frac{1}{\sqrt{\pi e}}\notag
  \end{align}
  by using the lemma below. 
 \end{proof}
%%%%%%%%%%
%%%%%%%%%%
 \begin{lem}
  Suppose there exists a positive constant $A>0$ such that 
  \begin{align}
   \lim_{t\downarrow 0}\frac{g(t)}{t^{1/2}}=A\notag
  \end{align}
  holds for a continuous function $g:[0,\infty)\rightarrow \R$. Then it holds 
  \begin{align}
   \lim_{t\downarrow 0}\frac{1}{t^{1/2}}\int_0^t\frac{ds}{g(s)}=\frac{2}{A}.\notag
  \end{align}
 \end{lem}
%%%%%%%%%%
 \begin{proof}
  By the assumption, for any $\varepsilon\in(0,1)$, there exists a $t_0>0$ such that 
  \begin{align}
   A(1-\varepsilon)<\frac{g(t)}{t^{1/2}}<A(1+\varepsilon)\notag
  \end{align}
  for any $t\in (0,t_0)$. Hence for $0<t<t_0$, we have 
  \begin{align} 
   \frac{1}{t^{1/2}}\int_0^t\frac{ds}{g(s)}>\frac{1}{t^{1/2}}\int_0^t\frac{1}{s^{1/2}}\frac{(1+\varepsilon)^{-1}}{A}\,ds=\frac{1}{t^{1/2}}\cdot \frac{(1+\varepsilon)^{-1}}{A}2t^{1/2}=\frac{2(1+\varepsilon)^{-1}}{A}.\notag
  \end{align}
  By the same argument, we also have 
  \begin{align}
  \frac{1}{t^{1/2}}\int_0^t\frac{ds}{g(s)}<\frac{1}{t^{1/2}}\int_0^t\frac{1}{s^{1/2}}\frac{(1-\varepsilon)^{-1}}{A}\,ds=\frac{1}{t^{1/2}}\cdot \frac{(1-\varepsilon)^{-1}}{A}2t^{1/2}=\frac{2(1-\varepsilon)^{-1}}{A}\notag
 \end{align}
 for any $0<t<t_0$. 
 Since $\varepsilon>0$ is arbitrary, one has 
 \begin{align}
  \limsup_{t\downarrow 0}=\liminf_{t\downarrow 0}=\frac{2}{A}.\notag
 \end{align}
 \end{proof}
%%%%%%%%%%
%%%%%%%%%%
 \begin{rem}[The case of $\RCD(0,N)$ spaces of dimension less than $N$]\label{rem:C0coll}
  Let $(X,d,\m)$ be an $\RCD(0,N)$ space of the dimension $k:=\dim_{ess}(X,d,\m)$ and assume $k<N$. Then by a similar calculation, we have the following; 
  For $\m$-a.e. $x\in \calR_k$, it holds  
  \begin{align}
   \lim_{t\rightarrow +0}\frac{F(t)}{t^{1-\frac{k}{2N}}}=\left(1-\frac{k}{2N}\right)^{-1}(4\pi e)^{-\frac{k}{2N}}\times \left(\frac{d\m\vert_{\calR_k}}{d\calH^k}(x)\right)^{-1/N}\in(0,+\infty).\notag
  \end{align}
  Hence $C_0(x)=0$ for a.e. $x\in X$ for $\RCD(0,N)$ spaces of $\dim_{ess}<N$.  
 \end{rem}
%%%%%%%%%%
%%%%%%%%%%
 \begin{rem}[The case of non-regular points]\label{rem:singpt}
  On the contrary of the above remark, the infinitesimal behavior of $F^x(t)$ can be $\sqrt{t}$ for a singular point $x$. We give several examples in the following; 
  \begin{enumerate}
   \item Consider the weighted space $([0,\infty),d_E,r^{N-1}dr)$. It is easy to find that $([0,\infty),d_E,r^{N-1}dr)$ is an $\RCD(0,N)$ space but it is \emph{not} non-collapsing. The heat kernel at $0$ is given by 
   \begin{align}
    p(0,x,t)=\frac{c_N}{\sqrt{t}}e^{-\frac{x^2}{2t}},\notag
   \end{align} 
   where $c_N$ is the normalized constant. Then by a simple calculation, we obtain 
   \begin{align}
    C_0(0)=2c_N^{-1/N}e^{-1/2}>0.\notag
   \end{align}
   \item On the upper-half space $(\R\times \R_{\geq0},d_E,\calH^2)$ which is a non-collapsing $\RCD(0,2)$ space, the heat kernel is written by 
   \begin{align}
    p(x,y,t)=(4\pi t)^{-1}\left[\exp\left(-\frac{\lv x-y\rv^2}{4t}\right)+\exp\left(-\frac{\lv x-y_*\rv^2}{4t}\right)\right]\notag
   \end{align}
   where $y_*=(y_1,-y_2)$ for $y=(y_1,y_2)$ (see \cite{IK}). Also by a simple calculation, we have 
   \begin{align}
    C_0((0,0))=\sqrt{\frac{2}{\pi e}}>0.\notag
   \end{align}  
   \item More generally, it is known in \cite{HP}*{Proposition 2.10} that the heat kernel of an $N$-metric measure cone $(X,d,\m)$ (see Definition \ref{def:mmcone}) is 
   \begin{align}
    p(x,y,t)=ct^{-N/2}\exp\left(-\frac{d(x,y)^2}{4t}\right),\notag
   \end{align}
   where $x$ is the pole and the normalized constant $c$ is 
   \begin{align}
    c=\frac{2^{1-N}}{N\m(B_1(x))\Gamma(N/2)}.\notag
   \end{align}
   Hence by a similar calculation, we obtain 
   \begin{align}
    C_0(x)=(\pi e)^{-1/2}\left(\frac{\m(B_1(x))}{\omega_N}\right)^{-1/N}.\notag
   \end{align} 
  \end{enumerate}
 \end{rem}
%%%%%%%%%%
For a given $\nu\in\calP_2(X)$, we define the set of \emph{barycenters} $\sfb(\nu)$ by 
\begin{align}
 \sfb(\nu)=\left\{x\in X\;;\;\int_{X}d^2(x,y)\nu(dy)=\min_{z\in X}\int_Xd^2(z,y)\nu(dy)\right\}.\notag
\end{align}
Note that $\sfb(\nu)\neq\emptyset$ for any $\nu\in\calP_2(X)$ by the properness of $(X,d)$. For generic metric space, minimizers may not be unique, but it does not matter in the proof later. Also we define 
\begin{align}
 \Var(\nu):=\min_{x\in X}\int_Xd^2(x,y)\nu(dy)\notag
\end{align} 
for $\nu\in\calP_2(X)$. 
%%%%%%%%%%
 \begin{thm}[Shannon's inequality on non-collapsing $\RCD(0,N)$ spaces]\label{thm:shannonRCD}
  Let $(X,d,\calH^N)$ be a non-collapsing $\RCD(0,N)$ space and $\nu\in\calD(\Ent_{\calH^N})$. Assume $\sfb(\nu)\cap\calR_N\neq \emptyset$. Then 
  \begin{align}
   U_N(\nu)\leq \sqrt{\frac{2\pi e}{N}\Var(\nu)}\label{eq:shannonnoncol}
  \end{align}
  holds. 
 \end{thm}
%%%%%%%%%%
 \begin{proof}
  Take a point $x\in \calR_N\cap \sfb(\nu)$. As the same proof in the case of $\R^N$, we have 
  \begin{align}
   U_N(\nu)\leq \frac{1}{F^x(t)}\left(t+\frac{\Var(\nu)}{2N}\right)\notag
  \end{align}
  for any $t>0$. Combining Lemma \ref{lem:key1} and Lemma \ref{lem:key2} leads 
  \begin{align}
   U_N(\nu)&\leq \frac{1}{F^x(\Var(\nu)/2N)}\left(\frac{\Var(\nu)}{2N}+\frac{\Var(\nu)}{2N}\right)\notag\\
   &\leq \sqrt{\pi e\frac{2N}{\Var(\nu)}}\times \frac{\Var(\nu)}{N}\notag\\
   &=\sqrt{\frac{2\pi e}{N}\Var(\nu)}.\notag
  \end{align}
 \end{proof}
%%%%%%%%%%
We give the following Corollaries by the same argument, thus we omit the proof. 
%%%%%%%%%%
\begin{cor}
 Let $(X,d,\m)$ be an $\RCD(0,N)$ space and $\nu\in\calP_2(X)$. Take a point $x\in \sfb(\nu)$ and assume there exists a positive constant $D_0>0$ such that 
 \begin{align}
  \lim_{t\downarrow 0}\frac{U_N(p(x,y,t)\m(dy))}{\sqrt{t}}=2D_0.\notag
 \end{align}
 Then we have 
 \begin{align}
  U_N(\nu)\leq \sqrt{\frac{2D_0^2}{N}\Var(\nu)}.\notag
 \end{align}
\end{cor}
%%%%%%%%%%
%%%%%%%%%%
 \begin{cor}\label{cor:weightedShannon}
  Let $(X,d,\beta\calH^N)$ be a non-collapsing $\RCD(0,N)$ space for $\beta>0$ and $\nu\in\calD(\Ent_{\beta\calH^N})$. Assume that $\sfb(\nu)\cap\calR_N\neq \emptyset$. Then 
  \begin{align}
   U_N(\nu)\leq \beta^{1/N}\sqrt{\frac{2\pi e}{N}\Var(\nu)}\notag
  \end{align} 
  holds. 
 \end{cor}
%%%%%%%%%%
%%%%%%%%%%%%%%%%%%%%%%%%%%%%%%%%%%%%%%%%%%%%%%%%%%%%%%%%%%%%%%%%%%%%%%%%%%%%%%%%%%%%%%%%%%%%%%%%%%
%
%  Section : Rigidity result		
%
%%%%%%%%%%%%%%%%%%%%%%%%%%%%%%%%%%%%%%%%%%%%%%%%%%%%%%%%%%%%%%%%%%%%%%%%%%%%%%%%%%%%%%%%%%%%%%%%%%
\section{Rigidity result}
In this section, we prove a rigidity result. 
In the beginning, we define cone metric measure spaces. 
%%%%%%%%%%
 \begin{defn}[$N$-metric measure cone \cite{Kcone,HP}]\label{def:mmcone}
  For $N\geq 2$, we define the \emph{$N$-metric measure cone} $(C(Y),d_{C(Y)},\m_{C(Y)})$ over an $\RCD(N-2,N-1)$ space $(Y,d_Y,\m_Y)$, whose diameter is at most $\pi$ if $N=2$, by 
  \begin{align}
   &C(Y):=[0,\infty)\times Y/(\{0\}\times Y),\notag\\
   &d_{C(Y)}((t_1,y_1),(t_2,y_2)):=\sqrt{t_1^2+t_2^2-2t_1t_2\cos(d_Y(y_1,y_2))},\notag\\
   &d\m_{C(Y)}(t,y):=t^{N-1}dt\otimes d\m_Y(y),\notag
  \end{align}
  where $dt$ is the 1-dimensional Lebesgue measure. $O_Y:=[(0,y)]$ is called the pole of $C(Y)$.  
 \end{defn}
%%%%%%%%%%
%%%%%%%%%%
\begin{rem}
 Since the diameter of $\RCD(N-2,N-1)$ spaces appeared in Definition \ref{def:mmcone} are at most $\pi$ (see \cite{St2,LV}), the metric $d_{C(Y)}$ is well-defined. It is known that $(C(Y),d_{C(Y)},\m_{C(Y)})$ is an $\RCD(0,N)$ space \cite{Kcone}. 
\end{rem}
%%%%%%%%%%
We assume that $(X,d,\m)$ is an $\RCD(0,N)$ space and with $\di(X,d)>0$. And we also assume the following; 
%%%%%%%%%%
\begin{assump}\label{assump:rigid}
 There exists a point $x_0\in X$ and a positive constant $D_0>0$ such that 
 \begin{align}
  U_N(\mu_t)=\sqrt{\frac{2D_0^2 }{N}\theta_t^2}\tag{$\ast$}
 \end{align}
 holds for any $t>0$, where 
 \begin{align}
  \mu_t(dy):=p(x_0,y,t)\m(dy),\quad \theta^2_t:=W_2^2(\mu_t,\delta_{x_0}).\notag
 \end{align} 
\end{assump} 
%%%%%%%%%%
%%%%%%%%%%
 \begin{rem}
  In the equality ($\ast$), we use $\theta_t^2$ instead of $\Var(\mu_t)$. On an $\RCD$ space $(X,d,\m)$, the set of barycenters of $p(x_0,y,t)\m(dy)$ does not always include $x_0$. See Appendix A.  
 \end{rem}
%%%%%%%%%%
%%%%%%%%%%
 \begin{lem}\label{lem:C0rigid}
  Under the Assumption \ref{assump:rigid}, it holds that 
  \begin{align}
   C_0=\frac{1}{D_0}\notag
  \end{align}
  where $C_0$ is the constant appeared in Lemma \ref{lem:key1}. 
 \end{lem}
%%%%%%%%%%
 \begin{proof}
  We have 
  \begin{align}
   F(t)&=\int_0^t\frac{ds}{U_N(\mu_s)}=\int_0^t\left(\sqrt{\frac{2D_0^2}{N}\theta_s^2}\right)^{-1}\,ds\notag\\
   &=\frac{1}{D_0}\int_0^t\sqrt{\frac{2Ns}{\theta_s^2}}\times \frac{ds}{2\sqrt{s}}\notag\\
   &\geq \frac{1}{D_0}\int_0^t\frac{1}{2\sqrt{s}}\,ds\notag\\
   &=\frac{\sqrt{t}}{D_0}\notag
  \end{align}
  by Lemma \ref{lem:theta}. Thus $C_0\geq D_0^{-1}$. 
  \par As in the first part of the proof of Theorem \ref{thm:shannonRCD} with (\ref{eq:UnRn}), for fixed $t_0>0$, we have 
  \begin{align}
   \sqrt{\frac{2D_0^2}{N}\theta^2_{t_0}}&=U_N(\mu_{t_0})\leq \left.\frac{1}{F(t)}\left(t+\frac{\theta_{t_0}^2}{2N}\right)\right\vert_{t=\frac{\theta^2_{t_0}}{2N}}\notag\\
   &=\frac{1}{F(\theta_{t_0}^2/2N)}\frac{\theta_{t_0}^2}{N}\notag\\
   &\leq \frac{1}{C_0}\sqrt{\frac{2N}{\theta_{t_0}^2}}\times \frac{\theta_{t_0}^2}{N}\notag\\
   &=\frac{1}{C_0}\sqrt{\frac{2\theta_{t_0}^2}{N}}.\label{eq:C0D0}
  \end{align}
  Hence $C_0\leq D_0^{-1}$. Therefore $C_0=D_0^{-1}$. 
 \end{proof}
%%%%%%%%%%
%%%%%%%%%%
 \begin{lem}\label{lem:thetaexplicit}
  Set $T_0:=\sup_{s>0}\theta^2_s/2N$. Under the Assumption \ref{assump:rigid}, it holds 
  \begin{align}
   \theta_t^2=2Nt\notag
  \end{align} 
  for any $t\in[0,T_0]$. 
 \end{lem}
%%%%%%%%%%
 \begin{proof}
  Take an arbitrary $t_0>0$ and fix it. By (\ref{eq:C0D0}) in the proof of Lemma \ref{lem:C0rigid}, we have 
  \begin{align}
   \sqrt{\frac{2D_0^2 \theta_{t_0}^2}{N}}&=U_N(\mu_{t_0})=\left.\frac{1}{F(t)}\left(t+\frac{\theta_{t_0}^2}{2N}\right)\right\vert_{t=\frac{\theta_{t_0}^2}{2N}}\notag\\
   &=\frac{1}{F(\theta_{t_0^2}/2N)}\frac{\theta_{t_0}^2}{N}.\notag
  \end{align}
  Accordingly 
  \begin{align}
   F\left(\frac{\theta_{t_0}^2}{2N}\right)=\sqrt{\frac{\theta_{t_0}^2}{2D_0^2 N}}.\notag
  \end{align}
  Also by the definition of $C_0=D_0^{-1}$, we have 
  \begin{align}
   %F\left(\frac{\theta_{t_0}^2}{2N}\right)\geq \frac{1}{D_0}\times \sqrt{\frac{\theta_{t_0}^2}{2N}}.\notag
   F(t)\geq \frac{1}{D_0}\sqrt{t}.\notag
  \end{align}
  Hence by the monotonicity of $t\mapsto F(t)/\sqrt{t}$, we have $D_0^{-1}\leq F(t)/\sqrt{t}\leq F(t_0)/\sqrt{t_0}=D_0$ for any $0<t\leq \theta_{t_0}^2/2N$. Thus we obtain 
  \begin{align}
   F(t)=D_0^{-1}\sqrt{t}\label{eq:Fanyt}
  \end{align}
   for any $t\leq \theta_{t_0}^2/2N$. Since $t_0$ is arbitrary, (\ref{eq:Fanyt}) holds for any $t\in[0,T_0]$. Therefore by the calculation in Lemma \ref{lem:C0rigid}, we have $\theta_t^2=2Nt$ for $t\in [0,T_0]$. 
 \end{proof}
%%%%%%%%%%
%%%%%%%%%%
 \begin{thm}\label{thm:conerigid}
  Under the Assumption \ref{assump:rigid}, we have that $(X,d,\m)$ is a metric measure cone. 
 \end{thm}
%%%%%%%%%%
 \begin{proof}
  %See \cite{Kcone} for the definition of cone metric measure spaces. 
  By the proof of Lemma \ref{lem:theta}, we have 
  \begin{align}
   \theta_t^2=2\int_0^t\int_Xp(x_0,y,s)d{\bf \Delta}\frac{d_{x_0}^2}{2}(y)\,ds.\notag
  \end{align}
  We denote the Radon measures by ${\bf \Delta}d_{x_0}^2/2=\mu+\nu$ where $\mu\ll\m$, $\nu$ is the singular part.
  \medskip
  \par \ul{{\bf Claim: $\nu=0$}}
  \begin{proof}[Proof of Claim]
   By the Laplacian comparison theorem, we have $\nu\leq 0$. For any bounded Borel set $A\subset X$ with $\m(A)=0$, it holds  
   \begin{align}
    \theta_t^2&=2\int_0^t\left(\int_Ap(x_0,y,s)d{\bf \Delta}\frac{d_{x_0}^2}{2}(y)+\int_{X\setminus A}p(x_0,y,s)d{\bf \Delta}\frac{d_{x_0}^2}{2}(y)\right)\,ds\notag\\
    &\leq 2\int_0^t\left(\int_Ap(x_0,y,s)d\nu(y)+N\int_{X\setminus A}p(x_0,y,s)d\m(y)\right)\,ds\notag\\
   &\leq 2\int_0^t\int_Ap(x_0,y,s)\,d\nu(y)ds+2N\int_0^t\int_Xp(x_0,y,s)\m(dy)\,ds\notag\\
   &=2\int_0^t\int_Ap(x_0,y,s)\,d\nu(y)ds+2Nt.\notag
   \end{align} 
   Hence by the assumption, we have 
   \begin{align}
    \int_0^t\int_Ap(x_0,y,s)\,d\nu(y)ds\geq 0.\notag
   \end{align} 
   Since $p(x_0,y,s)>0$ for any $y\in A$, $s\in(0,T_0)$ by the heat kernel estimate (\ref{eq:hkestimate}), we have $\nu(A)\geq 0$. By the arbitrariness of $A$, we obtain $\nu\geq0$. We have both $\nu\leq 0$ and $\nu\geq 0$. Thus we conclude $\nu=0$. 
  \end{proof}
  For $t\in(0,T_0)$, suppose there exists a Borel measurable set $A\subset X$ with $\m(A)>0$ such that 
  \begin{align}
   g< N\quad\text{on A}\notag
  \end{align}
  where $g$ is the Radon-Nikodym derivative of ${\bf \Delta}d_{x_0}^2/2$ with respect to $\m$. By the heat kernel estimate (\ref{eq:hkestimate}), $p(x_0,y,t)>0$ for $y\in A$. Thus we have  
  \begin{align}
   \theta_t^2&=2\int_0^t\left(\int_Ap(x_0,y,s)d{\bf \Delta}\frac{d_{x_0}^2}{2}(y)+\int_{X\setminus A}p(x_0,y,s)d{\bf \Delta}\frac{d_{x_0}^2}{2}(y)\right)\,ds\notag\\
   &<2\int_0^t\int_Xp(x_0,y,s)N\m(dy)\,ds\notag\\
   &=2Nt=\theta_t^2\notag
  \end{align}
  by Lemma \ref{lem:thetaexplicit}. This is a contradiction. 
  Hence we obtain  
  \begin{align}
   {\bf \Delta}\frac{d_{x_0}^2}{2}=N\m.\notag
  \end{align}
  This means that $d_{x_0}^2/2\in\calD(\Delta)$ and $\Delta d_{x_0}^2/2=N$. By Theorem 6.1 in \cite{GV} (see also Theorem 2.8 in \cite{HP}), we conclude the metric measure space $(X,d,\m)$ is an $N$-metric measure cone. 
 \end{proof}
%%%%%%%%%%
%%%%%%%%%%
 \begin{rem}\label{rem:attainnotheat}
  Let $(X,d,\m)$ be an $\RCD(0,N)$ space, which is not one point. Suppose there exists a probability measure $\nu\in \calP_2(X)$ with $x_0\in \sfb(\nu)$ such that 
  \begin{align}
   \left\{
   \begin{aligned}
   \lim_{t\rightarrow +0}\frac{F(t)}{\sqrt{t}}&=D_0^{-1}>0\notag\\
   U_N(\nu)&=\sqrt{\frac{2D_0^2}{N}\Var(\nu)},
   \end{aligned}\tag{$\spadesuit$}
   \right.
  \end{align}
  for a positive constant $D_0>0$, where 
  \begin{align}
   F(t)&:=\int_0^t\frac{ds}{U_N(\mu_s)},\notag\\
   \mu_t(dy)&:=p(x_0,y,t)\m(dy).\notag
  \end{align}
  Set $t_0:=\Var(\nu)/2N$. Then by a similar calculation in (\ref{eq:thetaW2}), we have 
  \begin{align}
   \sqrt{\frac{2D_0^2}{N}\Var(\nu)}&=U_N(\nu)\leq\frac{1}{F(t_0)}\left(t_0+\frac{1}{2N}\left(\Var(\nu)-W_2^2(\mu_{t_0},\nu)\right)\right)\notag\\
   &\leq \frac{D_0}{\sqrt{t_0}}\left(\frac{\Var(\nu)}{N}-W_2^2(\mu_{t_0},\nu)\right)\notag\\
   &=\sqrt{\frac{2D_0^2}{N}\Var(\nu)}-\sqrt{\frac{2ND_0^2}{\Var(\nu)}}W_2^2(\mu_{t_0},\nu).\notag
  \end{align}
  Hence $\nu=\mu_{t_0}$ and $t_0=\Var(\nu)/2N$. In this case, 
  \begin{align}
   2Nt_0=\Var(\nu)=\Var(\mu_{t_0})\leq W_2^2(\mu_{t_0},\delta_{x_0})\leq 2Nt_0. \notag
  \end{align}
  Therefore $\theta_{t_0}^2=2Nt_0$ and $x_0\in\sfb(\mu_{t_0})$. 
  \smallskip
  \par In general, we have $\theta_t^2\leq 2Nt$. Accordingly by Theorem \ref{thm:shannonRCD}, it holds 
  \begin{align}
   U_N(\mu_t)\leq \sqrt{\frac{2D_0^2}{N}W_2^2(\mu_{t_0},\delta_{x_0})}\leq 2D_0\sqrt{t}.\notag
  \end{align}
  Suppose there is an interval $I\subset [0,t_0]$ with $\calL^1(I)>0$ such that $\theta_t^2<2Nt$ on $I$. Then for any $t\in I$, we have $U_N(\mu_t)<2D_0\sqrt{t}$ and 
  \begin{align}
   F(t_0)=\int_0^{t_0}\frac{ds}{U_N(\mu_s)}>\frac{1}{2D_0}\int_0^{t_0}\frac{ds}{\sqrt{s}}=\frac{\sqrt{t_0}}{D_0}.\notag
  \end{align}
  Hence we obtain 
  \begin{align}
   U_N(\mu_{t_0})&\leq \frac{1}{F(t_0)}\left(t_0+\frac{W_2^2(\mu_{t_0},\delta_{x_0})}{2N}\right)\notag\\
   &<\frac{D_0}{\sqrt{t_0}}\frac{W_2^2(\mu_{t_0},\delta_{x_0})}{N}\notag\\
   &=\sqrt{\frac{2D_0^2}{N}W_2^2(\mu_{t_0},\delta_{x_0})}=U_N(\mu_{t_0}).\notag
  \end{align}
  This is a contradiction. Thus $\theta_t^2=2Nt$ for any $t\in[0,t_0]$. Along the proof of Theorem \ref{thm:conerigid}, we know that $(X,d,\m)$ is an $N$-metric measure cone. 
 % \smallskip 
  \par In Lemma \ref{lem:C0rigid}, we prove $D_0^{-1}=C_0$ under ($\ast$). However in ($\ast$), the relative entropy is described in terms of $W_2^2(\mu_{t_0}^{x_0},\delta_{x_0})$ not by $\Var(\mu_{t_0}^{x_0})$. Hence two conditions are bit different. In fact, ($\spadesuit$) is a stronger condition than ($\ast$). See Appendix A more precisely. 
 \end{rem}
%%%%%%%%%%
By Lemma \ref{lem:key2}, $C_0(x_0)=(\pi e)^{-1/2}$ at $x_0\in\calR_N$ for non-collapsing $\RCD(0,N)$ spaces. So we have the following corollary.
%%%%%%%%%%
 \begin{cor}\label{cor:regularcone}
  Let $(X,d,\calH^N)$ be a non-collapsing $\RCD(0,N)$ space. Assume there exists a point $x_0\in\calR_N$ such that 
  \begin{align}
   U_N(\mu_t)=\sqrt{\frac{2\pi e}{N}W_2^2(\mu_t,\delta_{x_0})}\notag
  \end{align}
  holds for all $t>0$, where $\mu_t(dy):=p(x_0,y,t)\m(dy)$. Then $(X,d,\calH^N)=(\R^N,d_E,\calL^N)$. 
 \end{cor}
%%%%%%%%%%
Since $\vol_g=\calH^n$ for a complete $n$-dimensional Riemannian manifold $(M,g)$, we also have the following corollary. 
%%%%%%%%%%
\begin{cor}
 Let $(M,d_g,\vol_g)$ be an $n$-dimensional complete Riemannian manifold with the Riemannian volume measure. Assume $\mathrm{Ric}_g\geq 0$ and there exists a point $x_0\in M$ such that 
  \begin{align}
   U_N(\mu_t)=\sqrt{\frac{2\pi e}{N}W_2^2(\mu_t,\delta_{x_0})}\notag
  \end{align}
  holds for all $t>0$, where $\mu_t(dy):=p(x_0,y,t)\vol_g(dy)$. Then $(M,d_g,\vol_g)=(\R^n,d_E,\calL^n)$.
  \end{cor}
%%%%%%%%%%
 Combining (c), (f), Remark \ref{rem:C0coll}, Corollary \ref{cor:regularcone}, we have the following. 
%%%%%%%%%%
 \begin{thm}\label{thm:equiv}
  Let $(X,d,\m)$ be an $\RCD(0,N)$ space. Then the followings are all equivalent. 
  \begin{enumerate}
   \item $N\in\bbN$ and $(X,d,\m)=(\R^N,d_E,\calL^N)$, 
   \item For $\m$-a.e. $x_0\in X$, ($\ast$) holds for all $t>0$, 
   \item For all $x_0\in X$, ($\ast$) holds for all $t>0$, 
   \item $(X,d,\m)$ is non-collapsing and Assumption \ref{assump:rigid} holds at $x_0\in \calR_N$ for all $t>0$. 
  \end{enumerate}
 \end{thm}
%%%%%%%%%%
%%%%%%%%%%
% \begin{rem}
%  In proving Theorem \ref{thm:equiv}, we are able to relax Assumption \ref{assump:rigid} to the following; there exist a point $x_0\in X$, a time $T_0>0$, and a positive constant $D_0>0$ such that ($\ast$) holds for any $0<t\leq T_0$ due to the monotonicity of $F(t)/\sqrt{t}$. The proof is straight forward, hence we omit it.     
% \end{rem}
%%%%%%%%%%
In the last of this section, we give an almost rigidity result. 
%%%%%%%%%%
 \begin{thm}[Almost rigidity theorem]
  For given $D_0,\varepsilon>0$, $N\in(1,\infty)$, there exists a constant $\delta=\delta(N,D_0,\varepsilon)>0$ such that if an $\RCD(0,N)$ space $(X,d,\m)$ and $x_0\in X$ satisfy $\m(B_1(x_0))=1$ and 
  \begin{align}
   \lv U_N(p(x_0,y,t)\m(dy))-\sqrt{\frac{2D_0^2}{N}W_2^2(p(x_0,y,t)\m(dy),\delta_{x_0})}\rv<\delta,\notag
  \end{align}
  for any $t>0$, then there exists an $\RCD(N-2,N-1)$ space $(Y,d_Y,\m_Y)$ such that  $d_{pmG}((X,d,\m,x_0),C(Y))<\varepsilon$, where $d_{pmG}$ is the pointed measured Gromov distance, and $C(Y)=(C(Y),d_{C(Y)},\m_{C(Y)},O_Y)$ is the $N$-metric measure cone over $Y$.    
 \end{thm}
%%%%%%%%%%
 \begin{proof}
  Suppose there exists a positive constant $\varepsilon>0$, a sequence of $\RCD(0,N)$ spaces $\{(X_n,d_n,\m_n)\}_n$, and $x_n\in X_n$ with $\m_n(B_1(x_n))=1$ such that it hold
  \begin{align}
   \lv U_N(p^{X_n}(x_n,y,t)\m_n(dy))-\sqrt{\frac{2D_0^2}{N}W_2^2(p^{X_n}(x_n,y,t)\m_n(dy),\delta_{x_n})}\rv<\frac{1}{n}\label{eq:almostcontr}
  \end{align}
  and the $d_{pmG}$-distance between those spaces and any $N$-metric measure cone away from at least $\varepsilon$. Since $\{(X_n,d_n,\m_n,x_n)\}_n$ is precompact in $d_{pmG}$ topology, we have a convergent subsequence $\{(X_{n_j},d_{n_j},\m_{n_j},x_{n_j})\}_j$ such that $(X_{n_j},d_{n_j},\m_{n_j},x_{n_j})\xrightarrow{pmG}(X,d,\m,x)$. By \cite{AHT}*{Theorem 3.3}, we have the pointwise convergence $p^{X_{n_j}}(x_{n_j},y_{n_j},t)\rightarrow p^X(x,y,t)$ whenever $(x_n,y_n,t)\rightarrow (x,y,t)$. Combining the above result and \cite{GMS}*{Theorem 7.7}, we obtain 
  \begin{align}
   &W_2^2(p^{X_{n_j}}(x_{n_j},y,t)\m_{n_j}(dy),\delta_{x_{n_j}})\rightarrow W_2^2(p^X(x,y,t)\m(dy),\delta_x),\notag\\
   &U_N(p^{X_{n_j}}(x_{n_j},y,t)\m_{n_j}(dy))\rightarrow U_N(p^{X}(x,y,t)\m(dy))\notag
  \end{align} 
  as $j\rightarrow \infty$. By (\ref{eq:almostcontr}) and Theorem \ref{thm:conerigid}, $(X,d,\m,x)$ is an $N$-metric measure cone. It implies $d_{pmG}((X_{n_j},d_{n_j},\m_{n_j},x_{n_j}),(X,d,\m,x))\geq \varepsilon$. This is a contradiction. 
 \end{proof}
%%%%%%%%%%
%%%%%%%%%%%%%%%%%%%%%%%%%%%%%%%%%%%%%%%%%%%%%%%%%%%%%%%%%%%%%%%%%%%%%%%%%%%%%%%%%%%%%%%%%%%%%%%%%%%%
%
%  section : uncertainty principle
%
%%%%%%%%%%%%%%%%%%%%%%%%%%%%%%%%%%%%%%%%%%%%%%%%%%%%%%%%%%%%%%%%%%%%%%%%%%%%%%%%%%%%%%%%%%%%%%%%%%%%
\section{Uncertainty principle}
In this section, we give a version of Heisenberg-Pauli-Weyl uncertainty principle inequality on $\RCD(0,N)$ spaces. Firstly, we recall the following result given in \cite{BKT}. 
\medskip
\par Let $(X,d,\m)$ be an $\RCD(0,N)$ space. We define 
\begin{align}
 \AVR_{\m}:=\lim_{r\rightarrow \infty}\frac{\m(B_r(x))}{\omega_N r^N},\notag
\end{align}
which takes the value in $[0,\infty)$ and independent of the choice of $x\in X$. Finiteness is guaranteed by the Bishop-Gromov inequality. In \cite{BKT}, the authors give the following result. 
%%%%%%%%%%%
 \begin{thm}[\cite{BKT}]\label{thm:ElogSob}
  Assume $\AVR_{\m}>0$. Then for any $u\in W^{1,2}(X,d,\m)$ with $\lV u\rV_2^2=1$, we have 
  \begin{align}
   \int_X\lv u\rv^2\log\lv u\rv^2\,d\m\leq \frac{N}{2}\log\left(\frac{2}{\pi e N}\AVR_{\m}^{-2/N}\int_X\lv \nabla u\rv^2\,d\m\right).\label{eq:ElogSob}
  \end{align}
 \end{thm}
%%%%%%%%%%% 
We have to remark that the authors in \cite{BKT} give the same results for more general settings. Here we just give the restricted form that we will use. 
\par The above result can be written in the different form used by the word of probability measures; Let $(X,d,\m)$ be an $\RCD(0,N)$ space with $\AVR_{\m}>0$ and $\nu\in \calD(\Ent_{\m})$ such that $\nu=\rho\m\ll\m$ and $\sqrt{\rho}\in \calD(\Ch)$. Then 
\begin{align}
 \Ent_{\m}(\nu)\leq \frac{N}{2}\log\left(\frac{1}{2\pi e N}\AVR_{\m}^{-2/N}I_{\m}(\nu)\right)\label{eq:ElogSobmeas}
\end{align} 
holds, where $I_{\m}(\nu)$ is the Fisher information of $\nu$ defined by 
\begin{align}
 I_{\m}(\nu):=4\int_{\{\rho>0\}}\lv \nabla \sqrt{\rho}\rv^2\,d\m=\int_{\{\rho>0\}} \frac{\lv\nabla\rho\rv^2}{\rho}\,d\m.\notag
\end{align}
Combining Theorem \ref{thm:ElogSob} and Theorem \ref{thm:shannonRCD} leads the following theorem. Note that the proof is similar those to \cite{Suguro,OS}.  
%%%%%%%%%%%
 \begin{thm}[Uncertainty principle on $\RCD(0,N)$ spaces]\label{thm:UPRCD}
  Let $(X,d,\calH^N)$ be a non-collapsing $\RCD(0,N)$ space. Assume $\AVR_{\calH^N}>0$ and $\nu\in\calD(\Ent_{\calH^N})\cap\calD(I_{\calH^N})$ such that $\sfb(\nu)\cap\calR_N\neq\emptyset$. Then 
  \begin{align}
   I_{\calH^N}(\nu)^{1/2}\Var(\nu)^{1/2}\geq N\AVR_{\calH^N}^{1/N}\label{eq:UPRCD}
  \end{align}
  holds. 
 \end{thm}
%%%%%%%%%%%
 \begin{proof}
  By (\ref{eq:ElogSobmeas}) and (\ref{eq:shannonnoncol}), we obtain 
  \begin{align}
   &-\frac{N}{2}\log\left(\frac{2\pi e}{N}\Var(\nu)\right)\leq \Ent_{\calH^N}(\nu)\leq \frac{N}{2}\log\left(\frac{1}{2\pi e N}\AVR_{\calH^N}^{-2/N}I_{\calH^N}(\nu)\right)\notag\\
   &\Rightarrow 0\leq \frac{N}{2}\log \left(\frac{1}{N^2}\AVR_{\calH^N}^{-2/N}\Var(\nu)I_{\calH^N}(\nu)\right)\notag\\
   &\Rightarrow 1\leq \frac{1}{N^2}\AVR_{\calH^N}^{-2/N}\Var(\nu)I_{\calH^N}(\nu).\notag
  \end{align}
 \end{proof}
%%%%%%%%%%%
Note that for every non-collapsing $(X,d,\beta\m)$ with $\AVR_{\beta\m}>0$, (\ref{eq:ElogSobmeas}) also holds; 
\begin{align}
 \Ent_{\beta\m}(\nu)=\Ent_{\m}(\nu)-\log\beta&\leq \frac{N}{2}\log\left(\frac{1}{2\pi e N}\AVR_{\m}^{-2/N}I_{\m}(\nu)\right)-\log\beta\notag\\
&= \frac{N}{2}\log\left(\frac{1}{2\pi e N}\AVR_{\beta\m}^{-2/N}I_{\beta\m}(\nu)\right).\notag
\end{align}
%%%%%%%%%%%
 \begin{cor}
  Let $(X,d,\beta\calH^N)$ be a non-collapsing $\RCD(0,N)$ space. Assume $\AVR_{\beta\calH^N}>0$ and $\nu\in\calD(\Ent_{\beta\calH^N})\cap \calD(I_{\beta\calH^N})$ such that $\sfb(\nu)\cap \calR_N\neq\emptyset$. Then 
  \begin{align}
   I_{\beta\calH^N}(\nu)^{1/2}\Var(\nu)^{1/2}\geq N\AVR_{\beta\calH^N}^{1/N}\beta^{-1/N}=N\AVR_{\calH^N}^{1/N}\notag
  \end{align}
  holds. 
 \end{cor}
%%%%%%%%%%%
 \begin{proof}
  This is a direct consequence from combining Corollary \ref{cor:weightedShannon} and the proof of Theorem \ref{thm:UPRCD}. The last equality comes from the identity $\AVR_{\beta\calH^N}=\beta\AVR_{\calH^N}$. 
 \end{proof}
%%%%%%%%%%%
%%%%%%%%%%%
\begin{rem}\label{rem:asympsharp}
 The inequality (\ref{eq:UPRCD}) is asymptotically sharp for a fixed space in the following sense; Let $(X,d,\calH^N)$ be a non-collapsing $\RCD(0,N)$ space with $\AVR_{\calH^N}>0$ and $x\in\calR_N$. For $r>0$, the scaled metric measure spaces 
 \begin{align}
  \{(X,d_r,(\AVR_{\calH^N})^{-1}\calH^{N,d_r})\}_{r>0}\notag
 \end{align}
 are non-collapsing $\RCD(0,N)$ spaces with $\AVR_{(X,d_r,(\AVR_{\calH^N})^{-1}\calH^{N,d_r})}=1$, here $\calH^{N,d_r}$ denotes the Hausdorff measure relative to the scaled distance function $d_r$. Thus there exists a subsequence $r_j\rightarrow \infty$ such that 
 \begin{align}
  \left(X,d_{r_j},(\AVR_{\calH^N})^{-1}\calH^{N,d_{r_j}},x\right)\xrightarrow{pmG} (Y,d_Y,\m_Y,y),\notag
 \end{align}
 which is called an asymptotic cone of $(X,d,\m,x)$. By the construction and \cite{DG}, $(Y,d_Y,\m_Y)$ is an $N$-metric measure cone with the pole $y$ and with $\AVR_{\m_Y}=1$ (see \cite{DG,KBishop,HP,Kcone}). By (3) of Remark \ref{rem:singpt}, we know the form of the heat kernel $p^{Y}$. By a simple calculation, $p^{Y}(y,z,t)\m_Y(dz)$ achieves the equality. 
 \par Hence the left-hand side of (\ref{eq:UPRCD}) approaches to the right-hand side if $\nu_j:=p(x,z,r_j^{-2}t)\calH^{N,d}(dz)$. 
\end{rem}
%%%%%%%%%%%
%%%%%%%%%%%
 \begin{rem}
  In Remark \ref{rem:asympsharp}, we see (\ref{eq:UPRCD}) is sharp for $N$-cone metric measure spaces. On the other hand, Han and Xu prove the rigidity theorem of uncertainty principle under $\mathsf{MCP}(0,N)$ condition, that is, under $\mathsf{MCP}(0,N)$ and essentially non-branching, (\ref{eq:UPRCD}) holds if and only if $(X,d,\m)$ is a volume cone. See \cite{HanX} for precise statement. It is worth pointing out that volume cone is metric cone under $\RCD(0,N)$ condition \cite{DGcone}.
 \end{rem}
%%%%%%%%%%%
Finally we give a version of uncertainty principle on $\RCD(K,N)$ space for $K>0$. The following is already known. 
%%%%%%%%%%%
\begin{thm}[$N$-Log Sobolev inequality (\cite{EKS}*{Corollary 3.29}, \cite{BGL}*{Theorem 6.8.1})]
  Let $(X,d,\m)$ be an $\RCD(K,N)$ space for $K>0$. Then for any $\nu\in\calP_2(X)$ 
  \begin{align}
   KN\left[\exp\left(\frac{2}{N}\Ent_{\m}(\nu)\right)-1\right]\leq I_{\m}(\nu).\label{eq:NlogSob}
  \end{align}
 \end{thm}
%%%%%%%%%%%
Since $\RCD(K,N)$ space for $K>0$ is an $\RCD(0,N)$ space, we obtain the following. 
%%%%%%%%%%%
 \begin{thm}[Uncertainty principle for positively curved spaces]
  Let $(X,d,\calH^N)$ be a non-collapsing $\RCD(K,N)$ space for $K>0$. Then for $\nu\in\calD(\Ent_{\calH^N})$ with $\sfb(\nu)\cap \calR_N\neq\emptyset$, 
  \begin{align}
   \Var(\nu)\left(1+\frac{1}{KN}I_{\calH^N}(\nu)\right)\geq \frac{N}{2\pi e}.\label{eq:uncertainKN}
  \end{align}
 \end{thm}
%%%%%%%%%%%
The proof is a similar way of that of Theorem \ref{thm:UPRCD}. Hence we omit the proof.

%%%%%%%%%%%%%%%%%%%%%%%%%%%%%%%%%%%%%%%%%%%%%%%%%%%%%%%%%%%%%%%%%%%%%%%%%%%%%%%%%%%%%%%%%%%%%%%%%%
%
%  Appendix : Barycenters of heat kernels
%
%%%%%%%%%%%%%%%%%%%%%%%%%%%%%%%%%%%%%%%%%%%%%%%%%%%%%%%%%%%%%%%%%%%%%%%%%%%%%%%%%%%%%%%%%%%%%%%%%%
\appendix
\section{Barycenters of heat kernels}
%%%%%%%%%%%%%%%%%%%%%%%%%%%%
\renewcommand{\theequation}{A.\arabic{equation}}
\setcounter{equation}{0}
%%%%%%%%%%%%%%%%%%%%%%%%%%%%
%%%%%%%%%%%%%%%%%%%%%%%%%%%%%%%%%%%%%%%%%%%%%%%%%%%%%%%%%%%%%%%%%%%%%%%%%%%%%%%%%%%%%%%%%%%%%%%%%%%%%%%%%%%%%%%%%%%%%%%%%%%%%%%%%%%%%%%%%%%%%%%%%%%
\subsection{Barycenters of the heat kernel of a non-collapsing space}
Set $X=[0,\pi]$, $d_X=d_E$ (Euclidean distance), and $\m_X(dx)=\calL^1(dx)=dx$ (1-dimensional Lebesgue measure on $X$). It is known that $(X,d_X,\m_X)$ is an $\RCD(0,1)$ space. On the product space $X\times [0,\pi]$, we define an equivalent relation $(x,t)\sim (x',s)$ if either $s=t=0$ or $s=t=\pi$ is held. We define 
\begin{align}
 &Y:=X\times [0,\pi]/_{\sim}\notag\\
 &d_Y((x,t),(x',s)):=\cos^{-1}\left(\cos(t)\cos(s)+\sin(t)\sin(s)\cos\left(d(x,x')\right)\right)\notag\\
 &\m_Y:=\frac{1}{2\pi}\sin t\,dxdt.\notag
\end{align}
Then a metric measure space $(Y,d_Y,\m_Y)$ is an $\RCD(1,2)$ space (see \cite{Kcone}). Note that $\RCD(1,2)$ space satisfies $\RCD(0,2)$ condition. One dimensional $\RCD$ spaces are classified, $[0,\ell]$ ($\ell>0$), $S^1(r)$ ($r>0$), $\R_{\geq0}$ or $\R$ (\cite{KL}). Hence $\dim_{ess}(Y,d_Y,\m_Y)=2$. Thus $(Y,d_Y,\m_Y)$ is a non-collapsing $\RCD(0,2)$ space (see \cite{BGHZ,Hnew} for more details). Note that $\m_Y$ is a Borel probability measure on $Y$. 
\medskip
\par Let $\bm{o}$ be an equivalent class of $X\times\{0\}$, one of two end points in $Y$. Since $d_Y(\bm{o},(x,t))=t$ for any $(x,t)\in Y$, 
\begin{align}
 W_2^2(\delta_{\bm{o}},\m_Y)=\frac{1}{2\pi}\int_0^{\pi}dx\int_0^{\pi}t^2\sin t\,dt=\frac{\pi^2-4}{2}.\label{eq:A1}
\end{align}
Take a point $y=(\pi/2,\pi/2)\in Y$. Then for any $(x,t)\in Y$, 
\begin{align}
 d_Y(y,\left(x,t\right))=\cos^{-1}\left(\sin t\sin x\right)\leq \frac{\pi}{2}-\sin t\sin x\notag
\end{align}
since $\sin t\sin x\geq 0$ and coming from the Taylor expansion of $\cos^{-1}$ at 0. Then we have 
\begin{align}
 W_2^2(\delta_y,\m_Y)&=\frac{1}{2\pi}\int_0^{\pi}\sin t\left(\int_0^{\pi}d_Y^2(y,(x,t))\,dx\right)\,dt\notag\\
 &\leq \frac{1}{2\pi}\int_0^{\pi}\sin t\left(\int_0^{\pi}\frac{\pi^2}{4}-\pi\sin t\sin x+\sin^2t\sin^2x\,dx\right)\,dt\notag\\
 &=\frac{\pi^2}{4}-\frac{\pi}{2}+\frac{1}{3}.\label{eq:A2}
\end{align} 
From (\ref{eq:A1}) and (\ref{eq:A2}), the difference between them is 
\begin{align}
 W_2^2(\delta_{\bm{o}},\m_Y)-W_2^2(\delta_{y},\m_Y)&\geq \frac{\pi^2-4}{2}-\left(\frac{\pi^2}{4}-\frac{\pi}{2}+\frac{1}{3}\right)\notag\\
 &=\frac{\pi^2}{4}+\frac{\pi}{2}-\frac{7}{3}\notag\\
 &\geq \frac{9}{4}+\frac{3}{2}-\frac{7}{3}>1>0.\notag
\end{align}
Consider $\mu_t(dy):=p(\bm{o},y,t)\m_Y(dy)$. It is known that $W_2(\mu_t,\m_Y)\rightarrow 0$ as $t\rightarrow \infty$. Hence for sufficiently large $t>0$, we have 
\begin{align}
 &\lv W_2^2(\delta_{\bm{o}},\mu_t)-W_2^2(\delta_{\bm{o}},\m_Y)\rv<\frac{1}{4}\notag\\
 &\lv W_2^2(\delta_y,\mu_t)-W_2^2(\delta_y,\m_Y)\rv<\frac{1}{4}.\notag
\end{align}
 Hence it holds 
\begin{align}
 W_2^2(\delta_{\bm{o}},\mu_t)-W_2^2(\delta_y,\mu_t)\geq W_2^2(\delta_{\bm{o}},\m_Y)-W_2^2(\delta_y,\m_Y)-\frac{1}{2}>\frac{1}{2}>0.\notag
\end{align}
This implies $\bm{o}\notin \sfb(\mu_t)$. 
%%%%%%%%%%%%%%%%%%%%%%%%%%%%%%%%%%%%%%%%%%%%%%%%%%%%%%%%%%%%%%%%%%%%%%%%%%%%%%%%%%%%%%%%%%%%%%%%%%%%%%%%%%%%%%%%%%%%%%%%%%%%%%%%%%%%%%%%%%%%%%%%%%%
\subsection{Barycenters of the heat kernels at the pole in $N$-metric measure cones}
In this subsection, we consider barycenters of the heat kernels at the pole in $N$-metric measure cones. Recall definition and the heat kernel in the following. 
\par Let $(Y,d_Y,\m_Y)$ be an $\RCD(N-2,N-1)$ space with $0<\di(Y,d_Y)\leq \pi$ for $N\geq 2$. Then $N$-metric measure cone $(C(Y),d_{C(Y)},\m_{C(Y)})$ over $(Y,d_Y,\m_Y)$ is defined by  
\begin{align}
 &C(Y):=[0,\infty)\times Y/(\{0\}\times Y),\notag\\
 &d_{C(Y)}((t_1,y_1),(t_2,y_2)):=\sqrt{t_1^2+t_2^2-2t_1t_2\cos(d_Y(y_1,y_2))},\notag\\
 &d\m_{C(Y)}(t,y):=t^{N-1}dt\otimes d\m_{Y}(y).\notag
\end{align}
We denote the pole of $C(Y)$ by $O:=[(0,Y)]$. Recall that 
\begin{enumerate}
 \item $(C(Y),d_{C(Y)},\m_{C(Y)})$ is an $\RCD(0,N)$ space,
 \item The heat kernel $p(O,(r,y),t)$ is written by 
 \begin{align}
  p(O,(r,y),t)=Ct^{-N/2}\exp\left(-\frac{d_{C(Y)}(O,(y,r))^2}{4t}\right), \notag
 \end{align}
 where 
 \begin{align}
  C=\frac{2^{1-N}}{N\m_{C(Y)}(B_1(O))\Gamma(N/2)}=\frac{2^{1-N}}{\m_Y(Y)\Gamma(N/2)}.\notag
 \end{align}
 By a simple calculation, we have (of course)
 \begin{align}
  &W_2^2(\mu_t,O)=2Nt,\notag\\
  &D_0=\sqrt{\pi e}\left(\frac{\m_{C(Y)}(B_1(O))}{\omega_N}\right)^{1/N}\notag\\
  &U_N(\mu_t)=2\sqrt{\pi et}\left(\frac{\m_{C(Y)}(B_1(O))}{\omega_N}\right)^{1/N}=\sqrt{\frac{2D_0^2}{N}W_2^2(\mu_t,O)},\notag
 \end{align} 
 where $d\mu_t(r,y)=p(O,(r,y),t)d\m_{C(Y)}(r,y)$. 
 Hence $N$-metric measure cone satisfies ($\ast$). 
 \smallskip
 \par Take a point $z=(r_0,y_0)\in C(Y)$. Then again by a calculation, we have 
 \begin{align}
  &W_2^2(\delta_z,\mu_t)=\int_{C(Y)}(r_0^2+r^2-r_0r\cos(d_Y(y_0,y)))p(O,(r,y),t)\,d\m_{C(Y)}(r,y)\notag\\
  &=r_0^2+2Nt-2r_0\int_{C(Y)}r\cos(d_Y(y_0,y))p(O,(r,y),t)\,d_{C(Y)}(r,y)\notag\\
  &=2Nt+r_0^2-2r_0\times 2\sqrt{t}\cdot\frac{\Gamma(N+1/2)}{\Gamma(N/2)}\cdot\frac{1}{\m_Y(Y)}\int_Y\cos(d_Y(y_0,y))\m_Y(dy)\notag\\
  &=2Nt+\left(r_0-2\sqrt{t}\cdot\frac{\Gamma(N+1/2)}{\Gamma(N/2)}\cdot\frac{1}{m_Y(Y)}\int_Y\cos(d_Y(y_0,y))\m_Y(dy)\right)^2\notag\\
  &-\frac{4t\Gamma(N+1/2)^2}{\Gamma(N/2)^2}\left(\frac{1}{\m_Y(Y)}\int_Y\cos(d_Y(y_0,y))\,\m_Y(dy)\right)^2.\notag
 \end{align}
 It is known (see the proof of Theorem 1.3 in \cite{ESrigid}) that 
 \begin{align}
  \int_Y\cos(d_Y(y_0,y))\,\m_Y(dy)\geq0.\notag
 \end{align}
 Hence $O\in \sfb(\mu_t)$ if and only if 
 \begin{align}
  \int_Y\cos(d_Y(y_0,y))\,\m_Y(dy)=0\tag{$\clubsuit$}
 \end{align} 
 for any $y_0\in Y$. 
 \smallskip
 \par The assumption ($\spadesuit$) implies $O\in \sfb(\mu_t)$ for small $t>0$. However, this assumption is equivalent to ($\clubsuit$). ($\clubsuit$) implies 
 \begin{align}
  \frac{1}{\m_Y(Y)^2}\int_Y\int_Y\cos(d_Y(y_0,y))\,\m_Y(dy_0)\m_Y(dy)=0.\tag{$\diamondsuit$}
 \end{align}
 By Corollary 1.4 in \cite{ESrigid}, ($\diamondsuit$) holds if and only if $Y$ is isomorphic to the sphere $\bbS^{N-1}$ with the round metric and a multiple of the volume measure. Therefore we have the following theorem. 
 %%%%%%%%%%%
  \begin{thm}\label{thm:A}
   Let $(X,d,\m)$ be an $\RCD(0,N)$ space, which is not one point. There exists a probability measure $\nu\in\calP_2(X)$ with $x_0\in \sfb(\nu)$ such that ($\spadesuit$) holds if and only if $X$ is isomorphic to the $N$-dimensional Euclidean space $\R^N$ with the Euclidean distance function and a multiple of the $N$-dimensional Lebesgue measure. 
  \end{thm}
 %%%%%%%%%%%
\end{enumerate}
Roughly speaking, it holds that 
\begin{align}
 (\spadesuit)\Leftrightarrow (X=\R^N)\Rightarrow (X\text{ is a cone })\Leftrightarrow (\ast).\notag
\end{align}

%%%%%%%%%%%%%%%%%%%%%%%%%%%%%%%%%%%%%%%%%%%%%%%%%%%%%%%%%%%%%%%%%%%%%%%%%%%%%%%%%%%%%%%%%%%%%%%%%%%%%%%%%%%%%%%%%%%%%%%%%%%%%%%%%%%%%%%%%%%%%%%%%%%
\section{Final remark}
Combining all results, finally we obtain the following theorem. 
%%%%%%%%%%%%
 \begin{thm}\label{thm:B}
  Let $(X,d,\calH^N)$ be a non-collapsing $\RCD(0,N)$ space and $\nu\in\calP_2(X)$. Then 
  \begin{align}
   -\Ent_{\calH^N}(\nu)\leq \frac{N}{2}\log\left(\frac{2\pi e}{N}\Var(\nu)\right)\notag
  \end{align}
  holds. 
 \end{thm}
%%%%%%%%%%%%
 \begin{proof}
  Take an arbitrary $\omega\in \calR_N$, and define $\mu_t^{\omega}(dx):=p(\omega,x,t)\calH^N(dx)$, $F^{\omega}(t)=\int_0^{t}U_N(\mu_s^{\omega})^{-1}\,ds$. By Lemma \ref{lem:key2} it holds $F^{\omega}(t)\geq \sqrt{t/\pi e}$. By (\ref{eq:thetaW2}), we have 
  \begin{align}
   U_N(\nu)&\leq \left.\frac{1}{F^{\omega}(t)}\left(t+\frac{1}{2N}W_2^2(\delta_{\omega},\nu)\right)\right\vert_{t=W_2^2(\delta_{\omega},\nu)/2N}\notag\\
   &\leq \sqrt{\pi e\cdot \frac{2N}{W_2^2(\delta_{\omega},\nu)}}\cdot \frac{W_2^2(\delta_{\omega},\nu)}{N}\notag\\
   &=\sqrt{\frac{2\pi e}{N}W_2^2(\delta_{\omega},\nu)}.\notag
  \end{align}
  Since $\calR_N$ is dense in $X$, we are able to take a sequence $\omega_n\in\calR_N$ such that $\omega_n\rightarrow \omega\in\sfb(\nu)$ as $n\rightarrow \infty$. Thus taking the limit $n\rightarrow \infty$, we obtain 
  \begin{align}
   U_N(\nu)&\leq \sqrt{\frac{2\pi e}{N}W_2^2(\delta_{\omega_n},\nu)}\notag\\
   &\rightarrow \sqrt{\frac{2\pi e}{N}W_2^2(\delta_{\omega},\nu)}=\sqrt{\frac{2\pi e}{N}\Var(\nu)}.\notag
  \end{align} 
 \end{proof}
%%%%%%%%%%%%
%%%%%%%%%%%%
 \begin{rem}
  Theorem \ref{thm:B} does not contradicts to Theorem 4.7 and Theorem \ref{thm:A}. Every $N$-metric measure cone $(C(Y),d_{C(Y)},\m_{C(Y)})$ over $\RCD(N-2,N-1)$ space $(Y,d_Y,\m_Y)$ satisfies 
  \begin{align}
   U_N(\mu_t)=\sqrt{\frac{2D_0^2}{N}W_2^2(\mu_t,O)},\notag
  \end{align}
  where 
  \begin{align}
   D_0=\sqrt{\pi e}\left(\frac{\m_{C(Y)}(B_1(O))}{\omega_N}\right)^{1/N}\notag
  \end{align}
  as seen in subsection A.2 (we use the same notation). If $(C(Y),d_{C(Y)},\m_{C(Y)})$ is non-collapsing, $D_0\leq \sqrt{\pi e}$ and the equality holds if and only if $O\in\calR_N$. Hence non-collapsing $N$-metric measure cone which is not isomorphic to $\R^N$ satisfies 
  \begin{align}
   U_N(\mu_t)=\sqrt{\frac{2D_0^2}{N}W_2^2(\mu_t,O)}\leq \sqrt{\frac{2\pi e}{N}\Var(\nu)}.\notag
  \end{align}  
 \end{rem}
%%%%%%%%%%%%

\section*{Acknowledgement}
The author would like to thank Professors Kazuhiro Kuwae, Masashi Misawa, Takeshi Suguro, Shouhei Honda, and Doctor Wataru Kumagai for their helpful comments and fruitful discussions. Especially he express gratitude to Shouhei Honda for pointing out Remark \ref{rem:asympsharp} and to Dr Hiroshi Tsuji for suggesting Theorem B.1. He is partly supported by JSPS KAKENHI Grant Numbers JP22K03291, and JP21K03238.   
\smallskip
\par {\bf Statements and Declarations.} No conflicts of interest or competing interests.


\begin{bibdiv}
\begin{biblist}

\bib{AGMR}{article}{
   author={Ambrosio, Luigi},
   author={Gigli, Nicola},
   author={Mondino, Andrea},
   author={Rajala, Tapio},
   title={Riemannian Ricci curvature lower bounds in metric measure spaces
   with $\sigma$-finite measure},
   journal={Trans. Amer. Math. Soc.},
   volume={367},
   date={2015},
   number={7},
   pages={4661--4701},
   issn={0002-9947},
   review={\MR{3335397}},
   doi={10.1090/S0002-9947-2015-06111-X},
}

\bib{AGS}{article}{
   author={Ambrosio, Luigi},
   author={Gigli, Nicola},
   author={Savar\'{e}, Giuseppe},
   title={Calculus and heat flow in metric measure spaces and applications
   to spaces with Ricci bounds from below},
   journal={Invent. Math.},
   volume={195},
   date={2014},
   number={2},
   pages={289--391},
   issn={0020-9910},
   review={\MR{3152751}},
   doi={10.1007/s00222-013-0456-1},
}

\bib{AGSmm}{article}{
   author={Ambrosio, Luigi},
   author={Gigli, Nicola},
   author={Savar\'{e}, Giuseppe},
   title={Metric measure spaces with Riemannian Ricci curvature bounded from
   below},
   journal={Duke Math. J.},
   volume={163},
   date={2014},
   number={7},
   pages={1405--1490},
   issn={0012-7094},
   review={\MR{3205729}},
   doi={10.1215/00127094-2681605},
}

%\bib{AHnew}{article}{
%   author={Ambrosio, Luigi},
%   author={Honda, Shouhei},
%   title={New stability results for sequences of metric measure spaces with
%   uniform Ricci bounds from below},
%   conference={
%      title={Measure theory in non-smooth spaces},
%   },
%   book={
%      series={Partial Differ. Equ. Meas. Theory},
%      publisher={De Gruyter Open, Warsaw},
%   },
%   date={2017},
%   pages={1--51},
%   review={\MR{3701735}},
%}

%\bib{AHloc}{article}{
%   author={Ambrosio, Luigi},
%   author={Honda, Shouhei},
%   title={Local spectral convergence in ${\rm RCD}^*(K,N)$ spaces},
%   journal={Nonlinear Anal.},
%   volume={177},
%   date={2018},
%   number={part A},
%   part={part A},
%   pages={1--23},
%   issn={0362-546X},
%   review={\MR{3865185}},
%   doi={10.1016/j.na.2017.04.003},
%}

%\bib{AHPT}{article}{
%   author={Ambrosio, Luigi},
%   author={Honda, Shouhei},
%   author={Portegies, Jacobus W.},
%   author={Tewodrose, David},
%   title={Embedding of ${\rm RCD}^{\ast}(K, N)$ spaces in $L^2$ via
%   eigenfunctions},
%   journal={J. Funct. Anal.},
%   volume={280},
%   date={2021},
%   number={10},
%   pages={Paper No. 108968, 72},
%   issn={0022-1236},
%   review={\MR{4224838}},
%   doi={10.1016/j.jfa.2021.108968},
%}


\bib{AHT}{article}{
   author={Ambrosio, Luigi},
   author={Honda, Shouhei},
   author={Tewodrose, David},
   title={Short-time behavior of the heat kernel and Weyl's law on ${\rm
   RCD}^*(K,N)$ spaces},
   journal={Ann. Global Anal. Geom.},
   volume={53},
   date={2018},
   number={1},
   pages={97--119},
   issn={0232-704X},
   review={\MR{3746517}},
   doi={10.1007/s10455-017-9569-x},
}

\bib{AMSbe}{article}{
   author={Ambrosio, Luigi},
   author={Mondino, Andrea},
   author={Savar\'{e}, Giuseppe},
   title={On the Bakry-\'{E}mery condition, the gradient estimates and the
   local-to-global property of $\mathsf{RCD}^*(K,N)$ metric measure spaces},
   journal={J. Geom. Anal.},
   volume={26},
   date={2016},
   number={1},
   pages={24--56},
   issn={1050-6926},
   review={\MR{3441502}},
   doi={10.1007/s12220-014-9537-7},
}

\bib{AMSnon}{article}{
   author={Ambrosio, Luigi},
   author={Mondino, Andrea},
   author={Savar\'{e}, Giuseppe},
   title={Nonlinear diffusion equations and curvature conditions in metric
   measure spaces},
   journal={Mem. Amer. Math. Soc.},
   volume={262},
   date={2019},
   number={1270},
   pages={v+121},
   issn={0065-9266},
   isbn={978-1-4704-3913-2},
   isbn={978-1-4704-5513-2},
   review={\MR{4044464}},
   doi={10.1090/memo/1270},
}


%\bib{BBI}{book}{
%   author={Burago, Dmitri},
%   author={Burago, Yuri},
%   author={Ivanov, Sergei},
%   title={A course in metric geometry},
%   series={Graduate Studies in Mathematics},
%   volume={33},
%   publisher={American Mathematical Society, Providence, RI},
%   date={2001},
%   pages={xiv+415},
%   isbn={0-8218-2129-6},
%   review={\MR{1835418}},
%   doi={10.1090/gsm/033},
%}


\bib{BKT}{article}{
   author={Balogh, Zolt\'{a}n M.},
   author={Krist\'{a}ly, Alexandru},
   author={Tripaldi, Francesca},
   title={Sharp log-Sobolev inequalities in $\mathsf{CD}(0,N)$ spaces with
   applications},
   journal={J. Funct. Anal.},
   volume={286},
   date={2024},
   number={2},
   pages={Paper No. 110217, 41},
   issn={0022-1236},
   review={\MR{4665494}},
   doi={10.1016/j.jfa.2023.110217},
}




\bib{BaSt}{article}{
   author={Bacher, Kathrin},
   author={Sturm, Karl-Theodor},
   title={Localization and tensorization properties of the
   curvature-dimension condition for metric measure spaces},
   journal={J. Funct. Anal.},
   volume={259},
   date={2010},
   number={1},
   pages={28--56},
   issn={0022-1236},
   review={\MR{2610378}},
   doi={10.1016/j.jfa.2010.03.024},
}

\bib{BGL}{book}{
   author={Bakry, Dominique},
   author={Gentil, Ivan},
   author={Ledoux, Michel},
   title={Analysis and geometry of Markov diffusion operators},
   series={Grundlehren der mathematischen Wissenschaften [Fundamental
   Principles of Mathematical Sciences]},
   volume={348},
   publisher={Springer, Cham},
   date={2014},
   pages={xx+552},
   isbn={978-3-319-00226-2},
   isbn={978-3-319-00227-9},
   review={\MR{3155209}},
   doi={10.1007/978-3-319-00227-9},
}

\bib{BGHZ}{article}{
   author={Brena, Camillo},
   author={Gigli, Nicola},
   author={Honda, Shouhei},
   author={Zhu, Xingyu},
   title={Weakly non-collapsed RCD spaces are strongly non-collapsed},
   journal={J. Reine Angew. Math.},
   volume={794},
   date={2023},
   pages={215--252},
   issn={0075-4102},
   review={\MR{4529413}},
   doi={10.1515/crelle-2022-0071},
}


\bib{BS}{article}{
   author={Bru\'{e}, Elia},
   author={Semola, Daniele},
   title={Constancy of the dimension for ${\rm RCD}(K,N)$ spaces via
   regularity of Lagrangian flows},
   journal={Comm. Pure Appl. Math.},
   volume={73},
   date={2020},
   number={6},
   pages={1141--1204},
   issn={0010-3640},
   review={\MR{4156601}},
   doi={10.1002/cpa.21849},
}

\bib{CMnew}{article}{
   author={Cavalletti, Fabio},
   author={Mondino, Andrea},
   title={New formulas for the Laplacian of distance functions and
   applications},
   journal={Anal. PDE},
   volume={13},
   date={2020},
   number={7},
   pages={2091--2147},
   issn={2157-5045},
   review={\MR{4175820}},
   doi={10.2140/apde.2020.13.2091},
}

\bib{CC1}{article}{
   author={Cheeger, Jeff},
   author={Colding, Tobias H.},
   title={On the structure of spaces with Ricci curvature bounded below. I},
   journal={J. Differential Geom.},
   volume={46},
   date={1997},
   number={3},
   pages={406--480},
   issn={0022-040X},
   review={\MR{1484888}},
}

\bib{CC2}{article}{
   author={Cheeger, Jeff},
   author={Colding, Tobias H.},
   title={On the structure of spaces with Ricci curvature bounded below. II},
   journal={J. Differential Geom.},
   volume={54},
   date={2000},
   number={1},
   pages={13--35},
   issn={0022-040X},
   review={\MR{1815410}},
}

\bib{CC3}{article}{
   author={Cheeger, Jeff},
   author={Colding, Tobias H.},
   title={On the structure of spaces with Ricci curvature bounded below.
   III},
   journal={J. Differential Geom.},
   volume={54},
   date={2000},
   number={1},
   pages={37--74},
   issn={0022-040X},
   review={\MR{1815411}},
}

\bib{CMglobal}{article}{
   author={Cavalletti, Fabio},
   author={Milman, Emanuel},
   title={The globalization theorem for the curvature-dimension condition},
   journal={Invent. Math.},
   volume={226},
   date={2021},
   number={1},
   pages={1--137},
   issn={0020-9910},
   review={\MR{4309491}},
   doi={10.1007/s00222-021-01040-6},
}

\bib{DGcone}{article}{
   author={De Philippis, Guido},
   author={Gigli, Nicola},
   title={From volume cone to metric cone in the nonsmooth setting},
   journal={Geom. Funct. Anal.},
   volume={26},
   date={2016},
   number={6},
   pages={1526--1587},
   issn={1016-443X},
   review={\MR{3579705}},
   doi={10.1007/s00039-016-0391-6},
}

\bib{DG}{article}{
   author={De Philippis, Guido},
   author={Gigli, Nicola},
   title={Non-collapsed spaces with Ricci curvature bounded from below},
   language={English, with English and French summaries},
   journal={J. \'{E}c. polytech. Math.},
   volume={5},
   date={2018},
   pages={613--650},
   issn={2429-7100},
   review={\MR{3852263}},
   doi={10.5802/jep.80},
}

\bib{DMR}{article}{
   author={De Philippis, Guido},
   author={Marchese, Andrea},
   author={Rindler, Filip},
   title={On a conjecture of Cheeger},
   conference={
      title={Measure theory in non-smooth spaces},
   },
   book={
      series={Partial Differ. Equ. Meas. Theory},
      publisher={De Gruyter Open, Warsaw},
   },
   date={2017},
   pages={145--155},
   review={\MR{3701738}},
}

\bib{DT}{article}{
   author={Dall'Ara, Gian Maria},
   author={Trevisan, Dario},
   title={Uncertainty inequalities on groups and homogeneous spaces via
   isoperimetric inequalities},
   journal={J. Geom. Anal.},
   volume={25},
   date={2015},
   number={4},
   pages={2262--2283},
   issn={1050-6926},
   review={\MR{3427124}},
   doi={10.1007/s12220-014-9512-3},
}

%\bib{DGcone}{article}{
%   author={De Philippis, Guido},
%   author={Gigli, Nicola},
%   title={From volume cone to metric cone in the nonsmooth setting},
%   journal={Geom. Funct. Anal.},
%   volume={26},
%   date={2016},
%   number={6},
%   pages={1526--1587},
%   issn={1016-443X},
%   review={\MR{3579705}},
%   doi={10.1007/s00039-016-0391-6},
%}

\bib{E}{article}{
   author={Erb, Wolfgang},
   title={Uncertainty principles on compact Riemannian manifolds},
   journal={Appl. Comput. Harmon. Anal.},
   volume={29},
   date={2010},
   number={2},
   pages={182--197},
   issn={1063-5203},
   review={\MR{2652457}},
   doi={10.1016/j.acha.2009.08.012},
}

\bib{EKS}{article}{
   author={Erbar, Matthias},
   author={Kuwada, Kazumasa},
   author={Sturm, Karl-Theodor},
   title={On the equivalence of the entropic curvature-dimension condition
   and Bochner's inequality on metric measure spaces},
   journal={Invent. Math.},
   volume={201},
   date={2015},
   number={3},
   pages={993--1071},
   issn={0020-9910},
   review={\MR{3385639}},
   doi={10.1007/s00222-014-0563-7},
}

\bib{ESrigid}{article}{
   author={Erbar, Matthias},
   author={Sturm, Karl-Theodor},
   title={Rigidity of cones with bounded Ricci curvature},
   journal={J. Eur. Math. Soc. (JEMS)},
   volume={23},
   date={2021},
   number={1},
   pages={219--235},
   issn={1435-9855},
   review={\MR{4186467}},
   doi={10.4171/jems/1010},
}

\bib{FS}{article}{
   author={Folland, Gerald B.},
   author={Sitaram, Alladi},
   title={The uncertainty principle: a mathematical survey},
   journal={J. Fourier Anal. Appl.},
   volume={3},
   date={1997},
   number={3},
   pages={207--238},
   issn={1069-5869},
   review={\MR{1448337}},
   doi={10.1007/BF02649110},
}

%\bib{Gsplit}{article}{
%   author={Gigli, Nicola},
%   title={An overview of the proof of the splitting theorem in spaces with
%   non-negative Ricci curvature},
%   journal={Anal. Geom. Metr. Spaces},
%   volume={2},
%   date={2014},
%   number={1},
%   pages={169--213},
%   review={\MR{3210895}},
%   doi={10.2478/agms-2014-0006},
%}


%\bib{GM}{article}{
%   author={Garofalo, Nicola},
%   author={Mondino, Andrea},
%   title={Li-Yau and Harnack type inequalities in $\mathsf{RCD}^*(K,N)$ metric
%   measure spaces},
%   journal={Nonlinear Anal.},
%   volume={95},
%   date={2014},
%   pages={721--734},
%   issn={0362-546X},
%   review={\MR{3130557}},
%   doi={10.1016/j.na.2013.10.002},
%}

\bib{G}{article}{
   author={Gigli, Nicola},
   title={On the differential structure of metric measure spaces and
   applications},
   journal={Mem. Amer. Math. Soc.},
   volume={236},
   date={2015},
   number={1113},
   pages={vi+91},
   issn={0065-9266},
   isbn={978-1-4704-1420-7},
   review={\MR{3381131}},
   doi={10.1090/memo/1113},
}

\bib{GMS}{article}{
   author={Gigli, Nicola},
   author={Mondino, Andrea},
   author={Savar\'{e}, Giuseppe},
   title={Convergence of pointed non-compact metric measure spaces and
   stability of Ricci curvature bounds and heat flows},
   journal={Proc. Lond. Math. Soc. (3)},
   volume={111},
   date={2015},
   number={5},
   pages={1071--1129},
   issn={0024-6115},
   review={\MR{3477230}},
   doi={10.1112/plms/pdv047},
}

\bib{GMos}{article}{
   author={Gigli, Nicola},
   author={Mosconi, Sunra},
   title={The abstract Lewy-Stampacchia inequality and applications},
   journal={J. Math. Pures Appl. (9)},
   volume={104},
   date={2015},
   number={2},
   pages={258--275},
   issn={0021-7824},
   review={\MR{3365829}},
   doi={10.1016/j.matpur.2015.02.007},
}

\bib{GP}{article}{
   author={Gigli, Nicola},
   author={Pasqualetto, Enrico},
   title={Behaviour of the reference measure on $\mathsf{RCD}$ spaces under
   charts},
   journal={Comm. Anal. Geom.},
   volume={29},
   date={2021},
   number={6},
   pages={1391--1414},
   issn={1019-8385},
   review={\MR{4367429}},
   doi={10.4310/CAG.2021.v29.n6.a3},
}

\bib{GPbook}{book}{
   author={Gigli, Nicola},
   author={Pasqualetto, Enrico},
   title={Lectures on nonsmooth differential geometry},
   series={SISSA Springer Series},
   volume={2},
   publisher={Springer, Cham},
   date={[2020] \copyright 2020},
   pages={xi+204},
   isbn={978-3-030-38612-2},
   isbn={978-3-030-38613-9},
   review={\MR{4321459}},
   doi={10.1007/978-3-030-38613-9},
}

\bib{GV}{article}{
   author={Gigli, Nicola},
   author={Violo, Ivan Yuri},
   title={Monotonicity formulas for harmonic functions in ${\rm RCD}(0,N)$
   spaces},
   journal={J. Geom. Anal.},
   volume={33},
   date={2023},
   number={3},
   pages={Paper No. 100, 89},
   issn={1050-6926},
   review={\MR{4533514}},
   doi={10.1007/s12220-022-01131-7},
}


\bib{HanX}{article}{
   author={Han, Bang-Xian},
   author={Xu, Zhe-Feng},
   title={Sharp uncertainty principles on metric measure spaces},
   journal={Calc. Var. Partial Differential Equations},
   volume={63},
   date={2024},
   number={4},
   pages={Paper No. 104, 16},
   issn={0944-2669},
   review={\MR{4732304}},
   doi={10.1007/s00526-024-02705-9},
}

\bib{Hnew}{article}{
   author={Honda, Shouhei},
   title={New differential operator and noncollapsed RCD spaces},
   journal={Geom. Topol.},
   volume={24},
   date={2020},
   number={4},
   pages={2127--2148},
   issn={1465-3060},
   review={\MR{4173928}},
   doi={10.2140/gt.2020.24.2127},
}

%\bib{HP}{article}{
%   author={Honda, Shouhei},
%   author={Peng, Yuanlin},
%   title={Sharp gradient estimate, rigidity and almost rigidity of Green functions on non-parabolic $\mathrm{RCD}(0,N)$ spaces},
%   journal={arXiv:2308.03974},
%   volume={27},
%   date={1948},
%   pages={379--423, 623--656},
%   issn={0005-8580},
%   review={\MR{26286}},
%   doi={10.1002/j.1538-7305.1948.tb01338.x},
%}


\bib{HP}{article}{
   author={Honda, Shouhei},
   author={Peng, Yuanlin},
   title={Sharp gradient estimate, rigidity and almost rigidity of Green functions on non-parabolic $\mathrm{RCD}(0,N)$ spaces},
   journal={Proc. Roy. Soc. Edinburgh Sect. A},
%%   volume={??},
%   date={2024},
%   number={??},
   pages={1--54},
%   issn={??},
%   review={??},
   doi={10.1017/prm.2024.131},
}


\bib{HKZ}{article}{
   author={Huang, Libing},
   author={Krist\'{a}ly, Alexandru},
   author={Zhao, Wei},
   title={Sharp uncertainty principles on general Finsler manifolds},
   journal={Trans. Amer. Math. Soc.},
   volume={373},
   date={2020},
   number={11},
   pages={8127--8161},
   issn={0002-9947},
   review={\MR{4169684}},
   doi={10.1090/tran/8178},
}

\bib{IK}{article}{
   author={Ishige, Kazuhiro},
   author={Kawakami, Tatsuki},
   title={Global solutions of the heat equation with a nonlinear boundary
   condition},
   journal={Calc. Var. Partial Differential Equations},
   volume={39},
   date={2010},
   number={3-4},
   pages={429--457},
   issn={0944-2669},
   review={\MR{2729307}},
   doi={10.1007/s00526-010-0316-4},
}

%\bib{J}{article}{
%   author={Jiang, Renjin},
%   author={Li, Huaiqian},
%   author={Zhang, Huichun},
%   title={Heat kernel bounds on metric measure spaces and some applications},
%   journal={Potential Anal.},
%   volume={44},
%   date={2016},
%   number={3},
%   pages={601--627},
%   issn={0926-2601},
%   review={\MR{3489857}},
%   doi={10.1007/s11118-015-9521-2},
%}

\bib{JLZ}{article}{
   author={Jiang, Renjin},
   author={Li, Huaiqian},
   author={Zhang, Huichun},
   title={Heat kernel bounds on metric measure spaces and some applications},
   journal={Potential Anal.},
   volume={44},
   date={2016},
   number={3},
   pages={601--627},
   issn={0926-2601},
   review={\MR{3489857}},
   doi={10.1007/s11118-015-9521-2},
}

\bib{Kcone}{article}{
   author={Ketterer, Christian},
   title={Cones over metric measure spaces and the maximal diameter theorem},
   language={English, with English and French summaries},
   journal={J. Math. Pures Appl. (9)},
   volume={103},
   date={2015},
   number={5},
   pages={1228--1275},
   issn={0021-7824},
   review={\MR{3333056}},
   doi={10.1016/j.matpur.2014.10.011},
}

\bib{KBishop}{article}{
   author={Kitabeppu, Yu},
   title={A Bishop-type inequality on metric measure spaces with Ricci
   curvature bounded below},
   journal={Proc. Amer. Math. Soc.},
   volume={145},
   date={2017},
   number={7},
   pages={3137--3151},
   issn={0002-9939},
   review={\MR{3637960}},
   doi={10.1090/proc/13517},
}

\bib{KL}{article}{
   author={Kitabeppu, Yu},
   author={Lakzian, Sajjad},
   title={Characterization of low dimensional $RCD^*(K,N)$ spaces},
   journal={Anal. Geom. Metr. Spaces},
   volume={4},
   date={2016},
   number={1},
   pages={187--215},
   review={\MR{3550295}},
   doi={10.1515/agms-2016-0007},
}

\bib{Kr}{article}{
   author={Krist\'{a}ly, Alexandru},
   title={Sharp uncertainty principles on Riemannian manifolds: the
   influence of curvature},
   language={English, with English and French summaries},
   journal={J. Math. Pures Appl. (9)},
   volume={119},
   date={2018},
   pages={326--346},
   issn={0021-7824},
   review={\MR{3862150}},
   doi={10.1016/j.matpur.2017.09.002},
}

%\bib{KOim}{article}{
%   author={Kombe, Ismail},
%   author={\"{O}zaydin, Murad},
%   title={Improved Hardy and Rellich inequalities on Riemannian manifolds},
%   journal={Trans. Amer. Math. Soc.},
%   volume={361},
%   date={2009},
%   number={12},
%   pages={6191--6203},
%   issn={0002-9947},
%   review={\MR{2538592}},
%   doi={10.1090/S0002-9947-09-04642-X},
%}

\bib{KM}{article}{
   author={Kell, Martin},
   author={Mondino, Andrea},
   title={On the volume measure of non-smooth spaces with Ricci curvature
   bounded below},
   journal={Ann. Sc. Norm. Super. Pisa Cl. Sci. (5)},
   volume={18},
   date={2018},
   number={2},
   pages={593--610},
   issn={0391-173X},
   review={\MR{3801291}},
}

\bib{KO}{article}{
   author={Kombe, Ismail},
   author={\"{O}zaydin, Murad},
   title={Hardy-Poincar\'{e}, Rellich and uncertainty principle inequalities on
   Riemannian manifolds},
   journal={Trans. Amer. Math. Soc.},
   volume={365},
   date={2013},
   number={10},
   pages={5035--5050},
   issn={0002-9947},
   review={\MR{3074365}},
   doi={10.1090/S0002-9947-2013-05763-7},
}

\bib{Li}{article}{
   author={Li, Zhenhao},
   title={The globalization theorem for ${\rm CD}(K, N)$ on locally finite
   spaces},
   journal={Ann. Mat. Pura Appl. (4)},
   volume={203},
   date={2024},
   number={1},
   pages={49--70},
   issn={0373-3114},
   review={\MR{4685719}},
   doi={10.1007/s10231-023-01352-9},
}


\bib{LV}{article}{
   author={Lott, John},
   author={Villani, C\'{e}dric},
   title={Ricci curvature for metric-measure spaces via optimal transport},
   journal={Ann. of Math. (2)},
   volume={169},
   date={2009},
   number={3},
   pages={903--991},
   issn={0003-486X},
   review={\MR{2480619}},
   doi={10.4007/annals.2009.169.903},
}

\bib{MM}{article}{
   author={Mart\'{\i}n, Joaquim},
   author={Milman, Mario},
   title={Isoperimetric weights and generalized uncertainty inequalities in
   metric measure spaces},
   journal={J. Funct. Anal.},
   volume={270},
   date={2016},
   number={9},
   pages={3307--3343},
   issn={0022-1236},
   review={\MR{3475459}},
   doi={10.1016/j.jfa.2016.02.016},
}

\bib{MN}{article}{
   author={Mondino, Andrea},
   author={Naber, Aaron},
   title={Structure theory of metric measure spaces with lower Ricci
   curvature bounds},
   journal={J. Eur. Math. Soc. (JEMS)},
   volume={21},
   date={2019},
   number={6},
   pages={1809--1854},
   issn={1435-9855},
   review={\MR{3945743}},
   doi={10.4171/JEMS/874},
}

\bib{OS}{article}{
   author={Ogawa, Takayoshi},
   author={Seraku, Kento},
   title={Logarithmic Sobolev and Shannon's inequalities and an application
   to the uncertainty principle},
   journal={Commun. Pure Appl. Anal.},
   volume={17},
   date={2018},
   number={4},
   pages={1651--1669},
   issn={1534-0392},
   review={\MR{3842878}},
   doi={10.3934/cpaa.2018079},
}

\bib{OST}{article}{
   author={Okoudjou, Kasso A.},
   author={Saloff-Coste, Laurent},
   author={Teplyaev, Alexander},
   title={Weak uncertainty principle for fractals, graphs and metric measure
   spaces},
   journal={Trans. Amer. Math. Soc.},
   volume={360},
   date={2008},
   number={7},
   pages={3857--3873},
   issn={0002-9947},
   review={\MR{2386249}},
   doi={10.1090/S0002-9947-08-04472-3},
}

\bib{RS}{article}{
   author={Rajala, Tapio},
   author={Sturm, Karl-Theodor},
   title={Non-branching geodesics and optimal maps in strong
   $CD(K,\infty)$-spaces},
   journal={Calc. Var. Partial Differential Equations},
   volume={50},
   date={2014},
   number={3-4},
   pages={831--846},
   issn={0944-2669},
   review={\MR{3216835}},
   doi={10.1007/s00526-013-0657-x},
}

\bib{Sh}{article}{
   author={Shannon, C. E.},
   title={A mathematical theory of communication},
   journal={Bell System Tech. J.},
   volume={27},
   date={1948},
   pages={379--423, 623--656},
   issn={0005-8580},
   review={\MR{26286}},
   doi={10.1002/j.1538-7305.1948.tb01338.x},
}

%\bib{St}{article}{
%   author={Stam, A. J.},
%   title={Some inequalities satisfied by the quantities of information of
%   Fisher and Shannon},
%   journal={Information and Control},
%   volume={2},
%   date={1959},
%   pages={101--112},
%   issn={0019-9958},
%   review={\MR{109101}},
%}


\bib{St1}{article}{
   author={Sturm, Karl-Theodor},
   title={On the geometry of metric measure spaces. I},
   journal={Acta Math.},
   volume={196},
   date={2006},
   number={1},
   pages={65--131},
   issn={0001-5962},
   review={\MR{2237206}},
   doi={10.1007/s11511-006-0002-8},
}

\bib{St2}{article}{
   author={Sturm, Karl-Theodor},
   title={On the geometry of metric measure spaces. II},
   journal={Acta Math.},
   volume={196},
   date={2006},
   number={1},
   pages={133--177},
   issn={0001-5962},
   review={\MR{2237207}},
   doi={10.1007/s11511-006-0003-7},
}

\bib{Stcor}{article}{
   author={Sturm, Karl-Theodor},
   title={Correction to ``On the geometry of metric measure spaces. I''},
   journal={Acta Math.},
   volume={231},
   date={2023},
   number={2},
   pages={387--390},
   issn={0001-5962},
   review={\MR{4683373}},
}


\bib{Suguro}{article}{
   author={Suguro, Takeshi},
   title={Shannon's inequality for the R\'{e}nyi entropy and an application to
   the uncertainty principle},
   journal={J. Funct. Anal.},
   volume={283},
   date={2022},
   number={6},
   pages={Paper No. 109566, 26},
   issn={0022-1236},
   review={\MR{4442594}},
   doi={10.1016/j.jfa.2022.109566},
}

%\bib{V}{book}{
%   author={Villani, C{\'e}dric},
%   title={Optimal transport},
%   series={Grundlehren der Mathematischen Wissenschaften [Fundamental
%   Principles of Mathematical Sciences]},
%   volume={338},
%   note={Old and new},
%   publisher={Springer-Verlag, Berlin},
%   date={2009},
%   pages={xxii+973},
%   isbn={978-3-540-71049-3},
%   review={\MR{2459454 (2010f:49001)}},
%   doi={10.1007/978-3-540-71050-9},
%}






\end{biblist}
\end{bibdiv}
\end{document}